\documentclass[hidelinks,onefignum,onetabnum]{siamart251216}

\usepackage{amssymb,mathtools}
\usepackage{enumitem}
\usepackage{mathrsfs}
\usepackage{microtype}

\numberwithin{theorem}{section}
\makeatletter
\newcommand{\siamthmsep}[2]{%
  \expandafter\let\csname #1\endcsname\relax
  \expandafter\let\csname end#1\endcsname\relax
  \expandafter\let\csname #1*\endcsname\relax
  \expandafter\let\csname end#1*\endcsname\relax
  \expandafter\let\csname the#1\endcsname\relax
  \theoremstyle{plain}\theoremheaderfont{\normalfont\sc}%
  \theorembodyfont{\normalfont\itshape}\theoremseparator{.}\theoremsymbol{}%
  \newtheorem{#1}{#2}[section]}
\newcommand{\siamremsep}[2]{%
  \expandafter\let\csname #1\endcsname\relax
  \expandafter\let\csname end#1\endcsname\relax
  \expandafter\let\csname #1*\endcsname\relax
  \expandafter\let\csname end#1*\endcsname\relax
  \expandafter\let\csname the#1\endcsname\relax
  \theoremstyle{plain}\theoremheaderfont{\normalfont\itshape}%
  \theorembodyfont{\normalfont}\theoremseparator{.}\theoremsymbol{}%
  \newtheorem{#1}{#2}[section]}
\makeatother
\siamthmsep{lemma}{Lemma}
\siamthmsep{proposition}{Proposition}
\siamthmsep{corollary}{Corollary}
\siamthmsep{definition}{Definition}
\siamthmsep{assumption}{Assumption}
\siamremsep{remark}{Remark}

\numberwithin{equation}{section}

\DeclareMathOperator{\Tr}{Tr}
\DeclareMathOperator{\dist}{dist}

\newcommand{\RR}{\mathbb R}
\newcommand{\EE}{\mathbb E}
\newcommand{\PP}{\mathbb P}
\newcommand{\cA}{\mathcal A}
\newcommand{\cD}{\mathcal D}

\newcommand{\cV}{\mathcal V}
\newcommand{\GammaD}{\Gamma}

\headers{Target-Generated Dirichlet Problems}{E. Bayraktar, H. Ding and I. Ekren}

\title{Target-Generated Dirichlet Problems in State-Constrained Stochastic Control%
\thanks{\textbf{Funding:} The first author was supported in part by National
Science Foundation grant DMS-2507940.  The third author was supported in part
by National Science Foundation grant DMS-2406240.}}
\author{Erhan Bayraktar\thanks{Department of Mathematics, University of Michigan,
530 Church Street, Ann Arbor, MI 48109-1043
(\email{erhan@umich.edu}, \email{hchding@umich.edu}, \email{iekren@umich.edu}).}
\and Haichuan Ding\footnotemark[2]
\and Ibrahim Ekren\footnotemark[2]}

\ifpdf
\hypersetup{
  pdftitle={Target-Generated Dirichlet Problems in State-Constrained Stochastic Control},
  pdfauthor={Erhan Bayraktar, Haichuan Ding and Ibrahim Ekren}
}
\fi

\begin{document}
\maketitle

\begin{abstract}
We study state-constrained stochastic control problems with a scalar surplus
$R\ge0$.  Controls for which the drift and volatility of $R$ vanish at zero
define a lower-dimensional Hamilton--Jacobi--Bellman equation.  Starting from
control-wise state-constraint viscosity inequalities, we prove that the upper
and lower limits of the interior value are a subsolution and a supersolution
of this boundary equation.  Terminal compatibility and comparison identify
their common limit.  We require one-sided bounds on positive surplus drift and generator
growth, together with domination by boundary generators. For compact
or coercive controls, these follow from a local lower bound and
continuity of the boundary control set. For a smooth stochastic target
value $w$, the transformation $R=Y-w(t,X)$ flattens the viable epigraph,
and, under the stated hypotheses, the boundary control value supplies
the Dirichlet datum. Applications include set-valued boundary controls,
unbounded drift with quadratic costs, and redundant hedging instruments
with a nonconstant target boundary.
\end{abstract}



\section{Introduction}

This paper studies the boundary values of state-constrained stochastic
control problems on the half-space $\RR^d\times[0,\infty)$. The controls
that keep the constrained coordinate equal to zero define a
lower-dimensional control problem on $\RR^d\times\{0\}$. We give
coefficient conditions under which the interior value converges to
the value of this boundary control problem, thereby converting the
state-constrained HJB problem into a Dirichlet problem. An important
source of these problems is a stochastic target constraint, in which
the boundary describes the minimal resource needed to meet a terminal
target.

The controlled state is $(X,R)$, where $X\in\RR^d$ and the scalar process $R$
is constrained to remain nonnegative; the Brownian motion and the controls
may be multidimensional.  The dynamics have the form
\[
\begin{aligned}
 dX_s &= b(s,X_s,R_s,a_s)\,ds+
        \sigma(s,X_s,R_s,a_s)\,dW_s,\\
 dR_s &= \mu(s,X_s,R_s,a_s)\,ds+
        \eta(s,X_s,R_s,a_s)\,dW_s,
\end{aligned}
\]
started from $(X_t,R_t)=(x,r)$.  A control is admissible only if
$R_s\ge0$ for all $s\in[t,T]$.
For $r>0$, the constrained value has the form
\[
  V(t,x,r)=\sup_a
  \EE\left[\int_t^T f(s,X_s,R_s,a_s)\,ds+\Phi(X_T,R_T)\right],
\]
where the supremum is over the state-constrained controls.  Dynamic programming
gives an HJB equation in the open set $r>0$:
\[
\begin{cases}
 -\partial_t u-H(t,x,r,Du,D^2u)=0,
      &(t,x,r)\in[0,T)\times\RR^d\times(0,\infty),\\
 u(T,x,r)=\Phi(x,r),
      &r>0.
\end{cases}
\]
Here $H(t,x,r,p,M)$ is the supremum over $a\in A$ of the generator of $(X,R)$
under the control $a$, applied to the jet $(p,M)$ in the variables $(x,r)$,
plus the running reward $f^a$; the equation involves $u_r$, $u_{rr}$, and the
mixed derivatives $D_xu_r$.  We seek the limit of $V$ as $r\downarrow0$.  The
controls that preserve the boundary form the set
\[
 A_0(t,x):=\{a\in A:\mu^a(t,x,0)=0,\ \eta^a(t,x,0)=0\}.
\]
Their remaining drift, covariance, and reward define the boundary Hamiltonian
$H_0$ in \eqref{eq:boundary-Hamiltonian} and the boundary control problem with
value $\cV(t,x)$.  Our theorem gives conditions under which
$\lim_{r\downarrow 0}V(t,x,r)=\cV(t,x)$, which supplies the endogenous Dirichlet condition
$u(t,x,0)=\cV(t,x)$ for the interior equation.  For example, multiplicative
surplus dynamics $dR_s=R_s\lambda_s\,ds+R_s\theta_s\,dW_s$ give
$\mu^a=r\lambda^a$ and $\eta^a=r\theta^a$, so the surplus drift and volatility
vanish on the state-constraint boundary.  In hedging problems with redundant
instruments, the surplus volatility fixes the aggregate hedge but leaves its
allocation among the instruments undetermined, while an excess return on the
underlying makes positive surplus drift available at the price of volatility.

Stochastic target problems provide a natural source of this half-space
structure. For controlled dynamics $(X,Y)$ as in
Section~\ref{sec:original-variables}, let $Y$ represent wealth, reserves,
or continuation utility. The minimal resource needed to meet the target
$Y_T\ge\Psi(X_T)$ is
\begin{equation}\label{eq:w}
  w(t,x):=\inf\{y:\exists \alpha
     \hbox{ such that }Y_T^{t,x,y,\alpha}\ge
       \Psi(X_T^{t,x,\alpha})\hbox{ a.s.}\}.
\end{equation}
When feasibility is monotone in the initial resource and the infimum
in \eqref{eq:w} is attained, the viable set is
$D=\{(t,x,y):y\ge w(t,x)\}$.

For a continuation problem constrained to this viable epigraph, start
from $Y_t=w(t,x)+r$ and introduce $R_s=Y_s-w(s,X_s)$, so that the
viability constraint becomes $R_s\ge0$. We call the resulting half-space
control problem \emph{target-generated} when the target and continuation
problems use the same controlled dynamics, control set, and admissibility
restrictions, possibly with different objectives.

\subsection{Literature positioning}

\paragraph{State constraints and stochastic targets}
Soner and Touzi introduced the dynamic programming and viscosity framework for
stochastic target problems \cite{SonerTouzi2002STP,SonerTouzi2002Flows}.

Bouchard, Elie and Imbert \cite[Section 3.3]{BEI2010} derive a Dirichlet
condition for a smooth target boundary under a unique locally Lipschitz
zero-volatility control (Assumption 3.3) and a local lower bound
(Assumption 3.4).  Here several controls may preserve the boundary, and the
boundary equation optimizes over all of them.  Our Corollary~\ref{cor:BEI}
recovers their invariant-boundary reduction from their full-Hamiltonian
subsolution inequality, which holds up to the boundary, and dynamic programming.

Bouchard, Gnecco Heredia, Moreau and Pham \cite{BGHMP2026} also derive a
boundary HJB by optimizing over a set of boundary-preserving controls, for an
expectation constraint with a compact state control and a uniformly elliptic
diffusion that yields a smooth boundary.  Their surplus drift is nonpositive
and the unbounded martingale control does not enter the state generator. Our Remark~\ref{rem:BGMP} recovers their Theorem 2.13 under their Assumption 2.10.

\paragraph{Fixed-domain Dirichlet theory}
Fixed-domain results provide the comparison and uniqueness background: the
controlled Markov process and viscosity framework of Fleming and Soner
\cite{FlemingSoner2006}, continuity of exit-time values under weak conditions
\cite{BayraktarSongYang2011}, identification of the exit-time value with a
Dirichlet HJB solution by stochastic Perron's method \cite{Rokhlin2014}, and
nonlinear Dirichlet solvability through a Perron supersolution
\cite{BayraktarSong2018}.  Our Section~\ref{sec:perron} also adapts stochastic Perron's method,
introduced for HJB equations by Bayraktar and S\^{\i}rbu
\cite{BayraktarSirbu2013}, to the boundary control problem.

\paragraph{Principal-agent problems}
Alvarez, Ekren and Huang \cite{AEH2026} provide a general approach to
principal-agent problems with adverse selection in which the credible
set is characterized through stochastic target problems. In the
risk-neutral, $X$-independent reduction of their Section~6, this set is
described by the lower and upper boundaries $\underline W(t)$ and
$\overline W(t)$. Our Corollary~\ref{cor:AEH} recovers their endpoint
equations by applying our theorem to each endpoint surplus. Our
Section~\ref{sec:examples} applies the same theorem to controlled drift
and hedging problems outside that principal-agent model.

\paragraph{Contribution}
The local lower bound of Bouchard, Elie and Imbert \cite{BEI2010} and the
signed endpoint estimates of Alvarez, Ekren and Huang \cite{AEH2026}
express a common restriction on positive surplus creation relative to
volatility: $\mu^+\le C(r+|\eta|+|\eta|^2)$ near the boundary.
At $r=0$, Lemma \ref{lem:drift-equivalence} equates this estimate with the
lower-bound component of Assumption 3.4 in \cite{BEI2010}; the endpoint
estimates in \cite{AEH2026} imply it nearby.  Theorem
\ref{thm:boundary-continuity} combines this restriction with generator bounds
to identify the boundary equation without requiring a unique boundary control.

\paragraph{Map of results and examples}
\begin{enumerate}[label=(\roman*)]
\item Theorem \ref{thm:boundary-continuity} proves, under the terminal limit and
comparison for the equation on $R=0$, that $V(t,x,r)\to v_0(t,x)$ as
$r\downarrow0$.
Under its additional feedback, growth, and stochastic sub- and supersolution
hypotheses, Theorem \ref{prop:perron} identifies the boundary control value
$\cV$ with that solution by stochastic Perron's method, so $v_0=\cV$.
\item For compact or coercive controls, Proposition \ref{prop:compact-coercive}
reduces the three coefficient conditions to the local lower bound of Bouchard, Elie and Imbert
\cite{BEI2010} and continuity of the boundary control set.  This case contains
their smooth invariant-boundary reduction.  Section \ref{sec:compact-controls}
and Proposition \ref{prop:redundant-hedging} have a continuum of boundary
controls, so their singleton assumption fails; the latter also has a
nonconstant target boundary.
\item For noncompact, noncoercive controls, the hypotheses are verified from
signed coefficient estimates, and Lemma \ref{prop:coefficient-set-test}
provides a coefficient-set test.  This covers the two endpoints in Alvarez,
Ekren and Huang \cite{AEH2026}, with the precise reduction in Corollary
\ref{cor:AEH}.  Section \ref{sec:noncompact-controls} gives an unbounded-control
example outside both the singleton and unbounded-control hypotheses of
Bouchard, Elie and Imbert \cite[Assumption 3.3(ii) and (3.11)]{BEI2010}.
\end{enumerate}

Section \ref{sec:failure-lower-bound} shows that the boundary identification
can fail if the local lower bound is omitted, even for a smooth target value.
The bound is not necessary for equality in every individual problem.

The rest of the paper is organized as follows.  Section
\ref{sec:half-space-problem} contains the results and proofs, Section
\ref{sec:original-variables} gives the verification from target dynamics,
Section \ref{sec:examples} presents the examples, and Section
\ref{sec:conclusion} concludes.  Appendix \ref{app:terminal-limit} proves the
terminal-limit criterion used in Section \ref{sec:half-space-problem}.

\section{The State-Constrained Half-Space Problem}\label{sec:half-space-problem}

\subsection{State constraint, boundary dynamics, and target generation}

We first state the result for a controlled half-space problem. In a
target-generated application, $R=Y-w(t,X)$ and $R=0$ is the image of the
graph $Y=w(t,X)$. The half-space domain is
$\cD=[0,T)\times\RR^d\times[0,\infty)$, with interior
$\cD^\circ=[0,T)\times\RR^d\times(0,\infty)$ and boundary
$\GammaD=[0,T)\times\RR^d\times\{0\}$.

Let $(\Omega,\mathcal F,\mathbb F,\PP)$ be a filtered probability space
satisfying the usual conditions, with
$\mathbb F=(\mathcal F_s)_{s\in[0,T]}$, and let $W$ be an
$\mathbb F$-Brownian motion. The controlled dynamics are
\begin{align}
 dX_s &= b(s,X_s,R_s,a_s)\,ds+
        \sigma(s,X_s,R_s,a_s)\,dW_s,\label{eq:X}\\
 dR_s &= \mu(s,X_s,R_s,a_s)\,ds+
        \eta(s,X_s,R_s,a_s)\,dW_s.\label{eq:R}
\end{align}
For an initial point $(t,x,r)$ with $r\ge0$, let $\cA_{t,x,r}$
be the family of $A$-valued, $\mathbb F$-progressively measurable
controls $a=(a_s)_{s\in[t,T]}$ for which the system
\eqref{eq:X}--\eqref{eq:R}, started from $(X_t,R_t)=(x,r)$,
has an $\mathbb F$-adapted continuous solution satisfying,
almost surely, $R_s\ge0$ for all $s\in[t,T]$.
For $r>0$, define the state-constrained value by
\begin{equation*}
  V(t,x,r)=\sup_{a\in\cA_{t,x,r}}
  \EE\left[\int_t^T f(s,X_s,R_s,a_s)\,ds+
  \Phi(X_T,R_T)\right].
\end{equation*}
Write $V^*$ and $V_*$ for the upper and lower semicontinuous envelopes obtained
by approaching the closure of the half-space from $r>0$.
The controls that keep the surplus equal to zero are
\begin{equation}\label{eq:A0}
 A_0(t,x):=
 \{a\in A:\eta(t,x,0,a)=0,\ \mu(t,x,0,a)=0\}.
\end{equation}
For
$a^0_s\in A_0(s,\bar X_s)$, the boundary dynamics are written as
\begin{equation}\label{eq:X-boundary}
 d\bar X_s=b(s,\bar X_s,0,a^0_s)\,ds+
        \sigma(s,\bar X_s,0,a^0_s)\,dW_s,
        \qquad a^0_s\in A_0(s,\bar X_s).
\end{equation}
Using the boundary dynamics \eqref{eq:X-boundary}, the boundary value is
\begin{equation}\label{eq:boundary-value}
  \cV(t,x)=\sup_{a^0\in\cA^0_{t,x}}
  \EE\left[\int_t^T f(s,\bar X_s,0,a^0_s)\,ds+
  \Phi(\bar X_T,0)\right].
\end{equation}
The desired endogenous Dirichlet condition is
$\lim_{r\downarrow0}V(t,x,r)=\cV(t,x)$; the rest of the section gives
conditions for it, and Section \ref{sec:original-variables} derives the
coefficients from a target problem in the variables $(X,Y)$.

\subsection{The boundary equation}

The main result identifies the boundary value from an upper bound on the
positive surplus drift, a growth bound on the generator, and domination of
the generators selected by the boundary penalization by the boundary
Hamiltonian.  At zero surplus, the drift condition is equivalent to the
local lower bound in Assumption 3.4 of Bouchard, Elie and Imbert
\cite{BEI2010}.

\begin{assumption}[Controlled diffusion]\label{ass:basic}
The control set $A$ is a metric space.  The coefficients
$b,\sigma,\mu,\eta$ and the rewards $f,\Phi$ are continuous.  The surplus coefficients $\mu,\eta$ are locally Lipschitz in $r$, uniformly
on compact sets of $(t,x,a)$.  The controlled state equations are well posed
and nonexplosive in the chosen formulation.  The rewards $f,\Phi$ are uniformly
continuous in the state variables on compact sets.  The admissible control
families are nonempty and stable under conditioning and concatenation.  The weak
dynamic programming principle holds.
\end{assumption}

For a smooth function $\psi=\psi(t,x,r)$, define the spatial generator
with running reward by
\[
\begin{aligned}
 (\mathscr L^a\psi)(t,x,r)
 &:=f(t,x,r,a)+b(t,x,r,a)\cdot D_x\psi(t,x,r)
   +\mu(t,x,r,a)\partial_r\psi(t,x,r)\\
 &\quad+\frac12\Tr\left[
      \begin{pmatrix}\sigma(t,x,r,a)\\ \eta(t,x,r,a)\end{pmatrix}
      \begin{pmatrix}\sigma(t,x,r,a)\\ \eta(t,x,r,a)\end{pmatrix}^{\!\top}
      D^2_{(x,r)}\psi(t,x,r)\right].
\end{aligned}
\]
\begin{assumption}[State-constraint viscosity inequalities]
\label{ass:state-constraint-viscosity}
All local extrema below are taken relative to $\cD$.
Let $\psi$ be a smooth test function such that
\begin{equation}\label{eq:concave-surplus-tests}
 \psi(t,x,r)=\phi(t,x)+\alpha r-\beta r^2,
 \qquad \phi\in C^{1,2},\quad \alpha\in\RR,\quad \beta>0,
\end{equation}
and $V^*-\psi$ has a local maximum at $(t,x,r)\in\cD$.
Then, for every $\delta>0$, there exists $a_\delta\in A$ such that
\[
 \bigl(-\partial_t\psi-\mathscr L^{a_\delta}\psi\bigr)(t,x,r)
 \le\delta.
\]
For any smooth test function $\psi$ such that $V_*-\psi$ has a
local minimum at $(t,x,r)\in\cD$, we have
\[
 \bigl(-\partial_t\psi-\mathscr L^a\psi\bigr)(t,x,r)\ge0
\]
for every $a\in A$ when $r>0$, and for every $a\in A_0(t,x)$
when $r=0$.
\end{assumption}

Assumption~\ref{ass:state-constraint-viscosity} identifies the viscosity
properties of $V^*$ and $V_*$ on $\cD$ used to derive the corresponding
properties of their boundary limits on $\GammaD$. This asymmetry is
consistent with Bouchard, Elie and Imbert \cite[Theorem~3.1]{BEI2010}:
the subsolution inequality uses the full Hamiltonian up to the boundary,
whereas the boundary supersolution inequality uses a constrained
operator. Neither assertion presupposes convergence of $V$ to $\cV$.

For $(t,x)\in[0,T]\times\RR^d$, define the limits from the open half-space by
\begin{align*}
 \overline V(t,x)&:=
   \limsup_{\substack{(s,y,r)\to(t,x,0)\\ r>0}}V(s,y,r),
   \\
 \underline V(t,x)&:=
   \liminf_{\substack{(s,y,r)\to(t,x,0)\\ r>0}}V(s,y,r).
\end{align*}
Equivalently, $\overline V(t,x)=V^*(t,x,0)$ and
$\underline V(t,x)=V_*(t,x,0)$.

The following elementary lemma records the contact points used in both parts of
the boundary proof.  The points may lie on $r=0$, where Assumption
\ref{ass:state-constraint-viscosity} supplies the constrained viscosity
inequality.

\begin{lemma}[Boundary contact points]\label{lem:boundary-contact}
Let $t_0<T$ and $\phi\in C^{1,2}$.
\begin{enumerate}[label=(\roman*)]
\item If $\overline V-\phi$ has a strict local maximum equal to zero at
$(t_0,x_0)$, then, for every $\varepsilon>0$, there are points
$(t_n,x_n,r_n)\in\cD$, with $0\le r_n<1/n$, such that
$V^*(t,x,r)-\phi(t,x)-\varepsilon(2nr-n^2r^2)$
has a local maximum at $(t_n,x_n,r_n)$ and
\begin{equation}\label{eq:upper-contact-points}
       (t_n,x_n)\longrightarrow(t_0,x_0),\qquad nr_n\longrightarrow0.
\end{equation}
\item If $\underline V-\phi$ has a strict local minimum equal to zero at
$(t_0,x_0)$, then there are points $(t_n,x_n,r_n)\in\cD$, with
$0\le r_n<1/n$, such that $V_*(t,x,r)-\phi(t,x)+nr$
has a local minimum at $(t_n,x_n,r_n)$ and
\begin{equation}\label{eq:lower-contact-points}
       (t_n,x_n)\longrightarrow(t_0,x_0),\qquad nr_n\longrightarrow0.
\end{equation}
\end{enumerate}
\end{lemma}

\begin{proof}
We prove (i); the argument for (ii) is the same with maxima replaced by minima.
Choose a compact cylinder $Q$ around $(t_0,x_0)$ on which the maximum is strict
and set $h(s)=2s-s^2$.  Maximize $V^*-\phi-\varepsilon h(nr)$ on
$Q\times[0,1/n]$; the maximum is at least zero since
$V^*(t_0,x_0,0)=\overline V(t_0,x_0)$.  If $(t_n,x_n,nr_n)\to(t,x,s)$ along a
subsequence, upper semicontinuity gives
$0\le\overline V(t,x)-\phi(t,x)-\varepsilon h(s)$, so strictness forces
$(t,x)=(t_0,x_0)$ and then $s=0$.  Thus every maximizing sequence satisfies
\eqref{eq:upper-contact-points} and stays away from the faces of
$Q\times[0,1/n]$ for large $n$.

For (ii), minimize $V_*-\phi+nr$ on $Q\times[0,1/n]$; any limit $(t,x,s)$ of
minimizers satisfies $0\ge\underline V(t,x)-\phi(t,x)+s$, so the strict
minimum forces $(t,x,s)=(t_0,x_0,0)$.
\end{proof}

For a test jet $(q,M)\in\RR^d\times\mathbb S^d$, write
\begin{equation*}
 \mathcal B^a(t,x,r;q,M)
 :=b(t,x,r,a)\cdot q+
   \frac12\Tr\big[\sigma\sigma^\top(t,x,r,a)M\big]
   +f(t,x,r,a).
\end{equation*}

\begin{assumption}[Controls on the boundary]\label{ass:boundary-controls}
The correspondence $A_0$ in \eqref{eq:A0} is nonempty and lower
semicontinuous on $[0,T)\times\RR^d$: if $(t_n,x_n)\to(t,x)$ with $t<T$
and $a\in A_0(t,x)$, there are
$a_n\in A_0(t_n,x_n)$ such that $a_n\to a$.  The boundary control family
$\cA^0_{t,x}$ is nonempty, and its admissible controls generate well-posed
dynamics \eqref{eq:X-boundary}.
\end{assumption}

The graph of $A_0$ is closed by continuity of $\mu$ and $\eta$.  When $A$ is
compact, Assumption \ref{ass:boundary-controls} therefore makes $A_0$ a
continuous compact-valued correspondence: $d_H(A_0(t_n,x_n),A_0(t,x))\to0$
whenever $(t_n,x_n)\to(t,x)$ with $t<T$, where, for subsets $C,D$ of $A$ with metric
$d_A$ and $\dist(c,D):=\inf_{d\in D}d_A(c,d)$,
\[
 e(C,D):=\sup_{c\in C}\dist(c,D),\qquad
 d_H(C,D):=\max\{e(C,D),e(D,C)\},
\]
with $e(\varnothing,D)=0$.
The boundary Hamiltonian is
\begin{equation}\label{eq:boundary-Hamiltonian}
 H_0(t,x,q,M):=\sup_{a\in A_0(t,x)}
       \mathcal B^a(t,x,0;q,M).
\end{equation}
In the compact-control case it is continuous by Berge's maximum theorem.

The following assumption bounds the positive part of the drift of $R$
in terms of its current value and the magnitude of its diffusion
coefficient near the boundary.
\begin{assumption}[Positive surplus drift]\label{ass:drift-cost}
For every compact $K\subset[0,T)\times\RR^d$ there are $C=C(K)$ and $r_0>0$
such that, for $(t,x)\in K$, $0\le r\le r_0$, and $a\in A$,
\begin{equation}\label{eq:drift-cost}
 \mu(t,x,r,a)^+\le C\bigl(r+|\eta(t,x,r,a)|+|\eta(t,x,r,a)|^2\bigr).
\end{equation}
\end{assumption}

\begin{assumption}[Generator growth]\label{ass:generator-growth}
For every compact $K\subset[0,T)\times\RR^d$, $Q<\infty$, and $\kappa>0$
there are $C_\kappa=C(K,Q,\kappa)$ and $r_0>0$ such that, for $(t,x)\in K$,
$0\le r\le r_0$, $a\in A$, and $|q|+|M|\le Q$,
\begin{equation}\label{eq:generator-growth}
 \mathcal B^a(t,x,r;q,M)\le C_\kappa+\kappa|\eta(t,x,r,a)|^2.
\end{equation}
\end{assumption}

\begin{assumption}[Boundary generators dominate]\label{ass:generator-domination}
The boundary Hamiltonian $H_0$ is finite and continuous on
$[0,T)\times\RR^d\times\RR^d\times\mathbb S^d$.  Moreover, if
$(t_n,x_n,r_n)\to(t,x,0)$ with $t<T$, $(q_n,M_n)\to(q,M)$, and $a_n\in A$
satisfy $\mu(t_n,x_n,r_n,a_n)\to0$ and $\eta(t_n,x_n,r_n,a_n)\to0$,
then
\begin{equation}\label{eq:sequential-generator-bound}
 \limsup_{n\to\infty}\mathcal B^{a_n}(t_n,x_n,r_n;q_n,M_n)\le H_0(t,x,q,M).
\end{equation}
\end{assumption}

All three conditions are one-sided: only the positive part of the surplus
drift is controlled, and \eqref{eq:generator-growth} and
\eqref{eq:sequential-generator-bound} are upper bounds.  Estimates of the form
$\mu^+\le C(r+|\eta|)$ and $\mathcal B^a\le C(1+|\eta|)$ imply
\eqref{eq:drift-cost} and \eqref{eq:generator-growth}, since
$C|\eta|\le C^2/(4\kappa)+\kappa|\eta|^2$.

The state, the Brownian motion, and the controls may be multidimensional; $\mu$
is scalar and $|\eta|$ is the Euclidean norm of the vector of surplus
volatilities.  

\paragraph{The local lower bound}
Bouchard, Elie and Imbert \cite{BEI2010} use, in their Assumption 3.4, the
condition
\begin{equation}\label{eq:local-lower-bound}
 \liminf_{\substack{\lambda\to\infty\\(s,y)\to(t,x)}}
 \inf_{a\in A}
 \left\{-\lambda\mu(s,y,0,a)+\lambda^2|\eta(s,y,0,a)|^2\right\}>-\infty
\end{equation}
at points $(t,x)\in[0,T)\times\RR^d$.  It is local in $(t,x)$ and is expressed
entirely in terms of the controlled coefficients on $R=0$.  In the variables
$R=Y-w(t,X)$ it reads
\begin{equation}\label{eq:bei-local-bound}
 \liminf_{\substack{n\to\infty\\(s,y)\to(t,x)}}
 \inf_{a\in A}
 \left\{n\bigl(L_X^a w-b_Y^a\bigr)
              +n^2|\sigma_Y^a-Dw\,\sigma_X^a|^2\right\}>-\infty ,
\end{equation}
which is the lower-bound component of their Assumption 3.4.  Their assumption
also invokes the unbounded-control condition (3.11) and states the lower bound
at points where the selected control is interior.  The next lemma identifies
\eqref{eq:local-lower-bound} with the boundary case of \eqref{eq:drift-cost}.
The signed drift estimates of Alvarez, Ekren and Huang \cite{AEH2026},
verified in Lemma \ref{lem:AEH-verification} below, imply this condition.

\begin{lemma}[Equivalent forms of the drift condition]\label{lem:drift-equivalence}
Fix $(t,x)\in[0,T)\times\RR^d$.  Condition \eqref{eq:local-lower-bound} holds
at $(t,x)$ if and only if there are a neighborhood $N$ of $(t,x)$ and a
constant $C$ such that
\begin{equation}\label{eq:drift-cost-boundary}
 \mu(s,y,0,a)^+\le C\bigl(|\eta(s,y,0,a)|+|\eta(s,y,0,a)|^2\bigr),
 \qquad (s,y)\in N,\ a\in A.
\end{equation}
In particular, Assumption \ref{ass:drift-cost} implies
\eqref{eq:local-lower-bound} at every point, and \eqref{eq:local-lower-bound}
on $[0,T)\times\RR^d$ is equivalent to \eqref{eq:drift-cost} restricted to
$r=0$.
\end{lemma}

\begin{proof}
Suppose \eqref{eq:local-lower-bound} holds.  There are $\lambda_0\ge1$, a
neighborhood $N$, and $C\ge1$ such that
$-\lambda\mu+\lambda^2|\eta|^2\ge-C$ for all $\lambda\ge\lambda_0$,
$(s,y)\in N$, and $a\in A$, the coefficients being evaluated at $(s,y,0,a)$.
Thus $\mu\le C/\lambda+\lambda|\eta|^2$ for all $\lambda\ge\lambda_0$.  If
$\eta=0$, letting $\lambda\to\infty$ gives $\mu\le0$.  If
$0<|\eta|\le\sqrt C/\lambda_0$, the choice $\lambda=\sqrt C/|\eta|\ge\lambda_0$
gives $\mu\le2\sqrt C\,|\eta|$.  If $|\eta|>\sqrt C/\lambda_0$, the choice
$\lambda=\lambda_0$ gives
$\mu\le C/\lambda_0+\lambda_0|\eta|^2\le\sqrt C\,|\eta|+\lambda_0|\eta|^2$.
This proves \eqref{eq:drift-cost-boundary}.

Conversely, suppose \eqref{eq:drift-cost-boundary} holds.  For
$\lambda\ge2C$,
\[
 -\lambda\mu+\lambda^2|\eta|^2
 \ge-C\lambda|\eta|-C\lambda|\eta|^2+\lambda^2|\eta|^2
 \ge\frac{\lambda^2}{2}|\eta|^2-C\lambda|\eta|
 \ge-\frac{C^2}{2},
\]
by minimizing the quadratic in $\lambda|\eta|$.  The last assertion of the
lemma follows by covering a compact set of $(t,x)$ by finitely many
neighborhoods $N$.
\end{proof}

By the local Lipschitz regularity in $r$ in
Assumption~\ref{ass:basic}, \eqref{eq:local-lower-bound} at every
$(t,x)$ is equivalent to Assumption~\ref{ass:drift-cost} when $A$
is compact. The same equivalence holds under
Assumption~\ref{ass:coercive-controls}; see
Proposition~\ref{prop:compact-coercive}.

\paragraph{Compact and coercive controls}
For compact or coercive control sets, Assumptions
\ref{ass:generator-growth} and \ref{ass:generator-domination} follow from
Assumption \ref{ass:boundary-controls}, and Assumption \ref{ass:drift-cost}
follows from \eqref{eq:local-lower-bound}.  The compactness needed in the
proof is thereby transferred from the controls to the generators they select.
The following is the standard coercivity condition used for unbounded controls
in Bouchard, Elie and Imbert \cite{BEI2010}.

\begin{assumption}[Coercivity]\label{ass:coercive-controls}
Let $A\subset\RR^m$ be closed.  For
every compact $K\subset[0,T)\times\RR^d$, every $Q<\infty$, and some $r_0>0$,
as $\rho\to\infty$,
\begin{align}
 &\inf_{\substack{(t,x)\in K,\ 0\le r\le r_0\\a\in A,\ |a|\ge \rho}}
       |\eta(t,x,r,a)|\longrightarrow\infty,
       \label{eq:eta-coercive}\\
 &\sup_{\substack{(t,x)\in K,\ 0\le r\le r_0,\ |q|+|M|\le Q\\
                   a\in A,\ |a|\ge \rho}}
 \frac{|\mu(t,x,r,a)|+|\mathcal B^a(t,x,r;q,M)|}
      {1+|\eta(t,x,r,a)|^2}
 \longrightarrow0.
 \label{eq:coercive-ratio}
\end{align}
\end{assumption}

\begin{proposition}[Compact or coercive controls]\label{prop:compact-coercive}
Suppose Assumptions \ref{ass:basic} and \ref{ass:boundary-controls} hold,
\eqref{eq:local-lower-bound} holds at every $(t,x)\in[0,T)\times\RR^d$, and
either $A$ is compact with coefficients and rewards locally bounded uniformly
in $a$, or Assumption \ref{ass:coercive-controls} holds.  Then Assumptions
\ref{ass:drift-cost}, \ref{ass:generator-growth}, and
\ref{ass:generator-domination} hold.
\end{proposition}

\begin{proof}
Fix a compact $K$ and $r_0>0$ as in Assumption \ref{ass:coercive-controls}
($r_0$ arbitrary in the compact case), and let $A_1$ be the set of controls
with $|\eta(t,x,r,a)|\le1$ for some $(t,x)\in K$, $0\le r\le r_0$; its closure
is compact, by \eqref{eq:eta-coercive} in the coercive case.  By Assumption \ref{ass:basic}, $\mu$ and $\eta$ are
Lipschitz in $r\in[0,r_0]$ with a constant $L$ uniform over
$K\times\overline{A_1}$; decrease $r_0$ so that $Lr_0\le1$.

\emph{Assumption \ref{ass:drift-cost}.}  By Lemma \ref{lem:drift-equivalence}
and compactness of $K$, \eqref{eq:drift-cost-boundary} holds on $K$ with a
constant $C$.  Let $(t,x)\in K$, $0\le r\le r_0$, $a\in A$, and put
$u=|\eta(t,x,r,a)|$.  If $u\le1$, then $a\in A_1$, and the Lipschitz bounds
$\mu(t,x,r,a)\le\mu(t,x,0,a)+Lr$ and $|\eta(t,x,0,a)|\le u+Lr$ give
\[
 \mu(t,x,r,a)^+\le C\bigl(u+Lr+2u^2+2L^2r^2\bigr)+Lr\le C'(r+u+u^2).
\]
If $u>1$, then $\mu(t,x,r,a)^+\le C''u^2$, where $C''$ bounds $|\mu|$ on
$K\times[0,r_0]\times A$ in the compact case; under \eqref{eq:coercive-ratio},
$|\mu|\le1+|\eta|^2\le2u^2$ for $|a|\ge\rho$, and $|\mu|$ is bounded on the
compact set $K\times[0,r_0]\times\{a\in A:|a|\le\rho\}$.

\emph{Assumption \ref{ass:generator-growth}.}  In the compact case
$\mathcal B^a$ is bounded for bounded jets; under \eqref{eq:coercive-ratio},
$|\mathcal B^a|\le\kappa(1+|\eta|^2)$ for $|a|\ge\rho_\kappa$ and
$\mathcal B^a$ is bounded for $|a|\le\rho_\kappa$.

\emph{Assumption \ref{ass:generator-domination}.}  Since
$A_0(t,x)\subset\overline{A_1}$ for $(t,x)\in K$, Assumption
\ref{ass:boundary-controls} and the closed graph make $A_0$ a continuous
correspondence with values in a fixed compact set, so $H_0$ is finite and
continuous on $K$ by Berge's maximum theorem.  Let $(t_n,x_n,r_n)\to(t,x,0)$,
$(q_n,M_n)\to(q,M)$, and $a_n\in A$ with
$\mu(t_n,x_n,r_n,a_n)\to0$ and $\eta(t_n,x_n,r_n,a_n)\to0$.  For large $n$,
$a_n\in A_1$.  Along any subsequence, extract $a_n\to a_0\in\overline{A_1}$.
Continuity gives $\mu(t,x,0,a_0)=\eta(t,x,0,a_0)=0$, so $a_0\in A_0(t,x)$,
and
$\mathcal B^{a_n}(t_n,x_n,r_n;q_n,M_n)\to\mathcal B^{a_0}(t,x,0;q,M)
\le H_0(t,x,q,M)$.  This proves \eqref{eq:sequential-generator-bound}.
\end{proof}

\paragraph{Noncompact controls}
When $A$ is neither compact nor coercive, the controls selected by the penalized
HJB need not have convergent subsequences, and
\eqref{eq:sequential-generator-bound} is verified by comparing generators
directly.  Write $C^a(t,x,r):=(b,\sigma\sigma^\top,f)(t,x,r,a)$
and, using the product Euclidean norm, define
\begin{align*}
 \mathfrak C_\delta(t,x,r)
   &:=\overline{\{C^a(t,x,r):a\in A,
             |\mu(t,x,r,a)|+|\eta(t,x,r,a)|\le\delta\}},\\
 \mathfrak C_0(t,x)
   &:=\overline{\{C^a(t,x,0):a\in A_0(t,x)\}}.
\end{align*}
For a set $E\subset\RR^d\times\mathbb S^d\times\RR$, define its support
function in the direction of a boundary jet by
\[
 h_E(q,M):=\sup_{(\beta,c,g)\in E}
 \left\{\beta\cdot q+\frac12\Tr(cM)+g\right\},
\]
with the supremum over the empty set equal to $-\infty$.  Then
$h_{\mathfrak C_0(t,x)}=H_0(t,x,\cdot)$.

\begin{lemma}[Coefficient-set test]
\label{prop:coefficient-set-test}
Suppose $H_0$ is finite and continuous and that, for every compact
$K\subset[0,T)\times\RR^d$ and every $Q<\infty$, there is a modulus
$\omega_{K,Q}(\epsilon)\downarrow0$ such that, whenever $(t,x)\in K$,
$0\le r\le\epsilon$, $0\le\delta\le\epsilon$, and $|q|+|M|\le Q$,
\begin{equation}\label{eq:support-function-bound}
 h_{\mathfrak C_\delta(t,x,r)}(q,M)
 \le H_0(t,x,q,M)+\omega_{K,Q}(\epsilon).
\end{equation}
Then Assumption \ref{ass:generator-domination} holds.  A sufficient condition
for \eqref{eq:support-function-bound} is
\begin{equation}\label{eq:coefficient-set-excess}
 \lim_{\epsilon\downarrow0}
 \sup_{\substack{(t,x)\in K\\0\le r\le\epsilon,\ 0\le\delta\le\epsilon}}
 e\bigl(\mathfrak C_\delta(t,x,r),\mathfrak C_0(t,x)\bigr)=0.
\end{equation}
\end{lemma}

\begin{proof}
Consider a sequence as in Assumption \ref{ass:generator-domination} and put
$\delta_n=|\mu(t_n,x_n,r_n,a_n)|+|\eta(t_n,x_n,r_n,a_n)|$ and
$\epsilon_n=\max\{r_n,\delta_n\}\to0$.  After restricting to a compact
neighborhood of $(t,x)$, \eqref{eq:support-function-bound} gives
\[
 \mathcal B^{a_n}(t_n,x_n,r_n;q_n,M_n)
 \le H_0(t_n,x_n,q_n,M_n)+\omega_{K,Q}(\epsilon_n),
\]
and continuity of $H_0$ proves \eqref{eq:sequential-generator-bound}.

It remains to verify the last assertion.  The linear functional defining
$h_E(q,M)$ has norm bounded by a constant $L_Q$ when $|q|+|M|\le Q$.
For $c\in\mathfrak C_\delta(t,x,r)$ and $\zeta>0$, choose
$c_0\in\mathfrak C_0(t,x)$ such that
$|c-c_0|\le\dist(c,\mathfrak C_0(t,x))+\zeta$.  The defining linear
functional, denoted by $\Lambda_{q,M}$, then satisfies
\[
 \Lambda_{q,M}(c)
 \le H_0(t,x,q,M)
     +L_Q\bigl(\dist(c,\mathfrak C_0(t,x))+\zeta\bigr).
\]
Take the supremum over $c$ and then let $\zeta\downarrow0$.
Condition \eqref{eq:coefficient-set-excess} now gives
\eqref{eq:support-function-bound}.
\end{proof}

For the boundary equation to determine the limit on the entire time interval,
the terminal value must also be attained when the terminal point is approached
from the interior.

\begin{assumption}[Terminal limit]\label{ass:terminal-limit}
For every $x\in\RR^d$,
\begin{equation}\label{eq:terminal-envelope-condition}
       \overline V(T,x)\le \Phi(x,0)\le \underline V(T,x).
\end{equation}
\end{assumption}

Since always $\underline V\le\overline V$, condition
\eqref{eq:terminal-envelope-condition} says precisely that the interior value
converges to $\Phi(x,0)$ at $(T,x,0)$.  Lemma \ref{lem:terminal-state-limit}
in Appendix \ref{app:terminal-limit} proves that it holds if $f$ and $\Phi$
are bounded, $\Phi$ is uniformly continuous, and,
locally uniformly in $x$,
\begin{equation}\label{eq:terminal-state-estimate}
 \lim_{\delta\downarrow0}
 \sup_{\substack{T-\delta\le t<T,\ |y-x|\le\delta\\0<r\le\delta}}
 \sup_{a\in\cA_{t,y,r}}
 \EE\left[|X_T^{t,y,r,a}-x|+R_T^{t,y,r,a}\right]=0.
\end{equation}
Bounded drift and diffusion coefficients give \eqref{eq:terminal-state-estimate} by the
standard short-time estimate.

We now state our main result.
\begin{theorem}[Continuity of the value at the boundary]
\label{thm:boundary-continuity}
Suppose Assumptions \ref{ass:basic}, \ref{ass:state-constraint-viscosity},
\ref{ass:boundary-controls}, \ref{ass:drift-cost}, \ref{ass:generator-growth},
and \ref{ass:generator-domination} hold, and $V$ is locally bounded up to
$r=0$ from the open half-space.
Then $\overline V$ is a viscosity subsolution
and $\underline V$ is a viscosity supersolution on
$[0,T)\times\RR^d$ of
\begin{equation}\label{eq:boundary-HJB}
 -\partial_t v-H_0(t,x,Dv,D^2v)=0.
\end{equation}
Suppose in addition that Assumption \ref{ass:terminal-limit} holds.  Then the
subsolution and supersolution satisfy the respective viscosity terminal
inequalities for $v(T,x)=\Phi(x,0)$.  If this terminal-value problem satisfies
comparison in a growth class containing $\overline V$ and $\underline V$, with
solution $v_0$, then
\begin{equation}\label{eq:boundary-envelope-equality}
 \overline V(t,x)=\underline V(t,x)=v_0(t,x).
\end{equation}
In particular, for every compact $K\subset[0,T)\times\RR^d$,
\begin{equation}\label{eq:local-uniform-boundary}
 \lim_{r\downarrow0}\sup_{(t,x)\in K}|V(t,x,r)-v_0(t,x)|=0.
\end{equation}
The same conclusions hold at the face $r=0$ of a strip
$0\le r\le\ell(t,x)$, provided $\ell$ is continuous and strictly positive
for $t<T$, and all envelopes and hypotheses are understood relative to the strip.  If the boundary control value $\cV$ in \eqref{eq:boundary-value} is a
viscosity solution of the same terminal-value problem in the comparison class,
then $v_0=\cV$.
The result also holds with $\RR^d$ replaced by an open state domain $O$, locally
uniformly on compact subsets of $[0,T)\times O$, provided comparison for the
boundary equation is formulated with the relevant conditions on $\partial O$.
\end{theorem}

\begin{proof}
\emph{Subsolution property.}
Let $\phi\in C^{1,2}$ and suppose that $\overline V-\phi$ has a strict local
maximum equal to zero at $(t_0,x_0)$, with $t_0<T$.  Fix $\varepsilon>0$ and
apply Lemma \ref{lem:boundary-contact}(i) to
$\phi_n(t,x,r)=\phi(t,x)+\varepsilon(2nr-n^2r^2)$,
which gives contact points satisfying \eqref{eq:upper-contact-points}.  Write
$q_n=D\phi(t_n,x_n)$ and $M_n=D^2\phi(t_n,x_n)$.  The state-constraint
subsolution inequality with $\delta=n^{-1}$ provides $a_n\in A$ such that,
with $\mu_n=\mu(t_n,x_n,r_n,a_n)$ and $\eta_n=\eta(t_n,x_n,r_n,a_n)$,
\begin{equation}\label{eq:penalized-subsolution}
 -\partial_t\phi(t_n,x_n)-\mathcal B^{a_n}(t_n,x_n,r_n;q_n,M_n)
 +\varepsilon P_n\le n^{-1},
\end{equation}
where $P_n:=-2n(1-nr_n)\mu_n+n^2|\eta_n|^2$.  Fix a compact neighborhood $K$ of $(t_0,x_0)$ containing the contact points
and $Q$ bounding $|q_n|+|M_n|$, and let $C$ be the constant of Assumption
\ref{ass:drift-cost} on $K$.  Since $0<1-nr_n\le1$, \eqref{eq:drift-cost}
gives, for $n\ge4C$,
\begin{equation}\label{eq:Pn-lower}
 P_n\ge n^2|\eta_n|^2-2Cn\bigl(r_n+|\eta_n|+|\eta_n|^2\bigr)
 \ge\frac{n^2}{4}|\eta_n|^2-4C^2-2Cnr_n ,
\end{equation}
by minimizing $\frac{n^2}{4}u^2-2Cnu$ over $u\ge0$.  As $nr_n\to0$, there is
$C_1=4C^2+1$, independent of $\varepsilon$, with
$P_n\ge\frac{n^2}{4}|\eta_n|^2-C_1$ for all large $n$.

Insert \eqref{eq:Pn-lower} and \eqref{eq:generator-growth} with
$\kappa=\varepsilon/8$ into \eqref{eq:penalized-subsolution}:
\[
 \varepsilon\Bigl(\frac{n^2}{4}|\eta_n|^2-C_1\Bigr)
 \le\sup_K|\partial_t\phi|+C_{\varepsilon/8}+\frac{\varepsilon}{8}|\eta_n|^2+1 .
\]
Hence $n^2|\eta_n|^2$ is bounded, so $\eta_n\to0$, and
\eqref{eq:penalized-subsolution} then bounds $P_n$ above by a constant
$C_\varepsilon$.  By \eqref{eq:drift-cost}, $\mu_n^+\to0$.  By the upper bound
on $P_n$, $2n(1-nr_n)\mu_n\ge n^2|\eta_n|^2-C_\varepsilon\ge-C_\varepsilon$,
so $\mu_n^-\to0$.  Thus $\mu_n\to0$ and $\eta_n\to0$.  Returning to
\eqref{eq:penalized-subsolution} with the lower bound $P_n\ge-C_1$,
\[
 -\partial_t\phi(t_n,x_n)-\mathcal B^{a_n}(t_n,x_n,r_n;q_n,M_n)
 \le\varepsilon C_1+n^{-1}.
\]
Since $\mu_n\to0$ and $\eta_n\to0$, Assumption
\ref{ass:generator-domination} applied at $(t_0,x_0)$ with the jets
$(q_n,M_n)\to(D\phi,D^2\phi)(t_0,x_0)$ gives
\[
 -\partial_t\phi(t_0,x_0)
 -H_0(t_0,x_0,D\phi(t_0,x_0),D^2\phi(t_0,x_0))\le\varepsilon C_1 .
\]
Letting $\varepsilon\downarrow0$ proves the subsolution inequality.

\emph{Supersolution property.}
Suppose that $\underline V-\phi$ has a strict local minimum equal to zero at
$(t_0,x_0)$.  Apply Lemma \ref{lem:boundary-contact}(ii) to the test functions
$\phi(t,x)-nr$.  The resulting contact points satisfy
\eqref{eq:lower-contact-points}.  Fix
$a_0\in A_0(t_0,x_0)$.  By Assumption \ref{ass:boundary-controls}, choose
$a_n\in A_0(t_n,x_n)$ with $a_n\to a_0$.  At points with $r_n>0$, the interior
supersolution inequality evaluated at $a_n$ gives the inequality below; at
points with $r_n=0$, it follows from the boundary part of Assumption
\ref{ass:state-constraint-viscosity}.  Thus
\[
 0\le -\partial_t\phi(t_n,x_n)
      -\mathcal B^{a_n}(t_n,x_n,r_n;D\phi,D^2\phi)
      +n\mu(t_n,x_n,r_n,a_n).
\]
Since $\mu(t_n,x_n,0,a_n)=0$, Lipschitz continuity and
$nr_n\to0$ imply that the last drift term tends to zero.  We obtain
\[
 -\partial_t\phi(t_0,x_0)
 -\mathcal B^{a_0}(t_0,x_0,0;D\phi,D^2\phi)\ge0.
\]
This holds for every $a_0\in A_0(t_0,x_0)$ and therefore gives the
supersolution inequality for \eqref{eq:boundary-HJB}.

Assumption \ref{ass:terminal-limit} gives the terminal inequalities.  When the
drift and diffusion coefficients and rewards are bounded, \eqref{eq:terminal-state-estimate} follows
from the standard short-time estimate, so the terminal-limit assumption follows
from the uniform continuity of $\Phi$.

Comparison yields $\overline V\le\underline V$, while the reverse inequality is
automatic.  Thus \eqref{eq:boundary-envelope-equality} holds.  Since $v_0$ is
continuous, a compactness argument gives the locally uniform convergence
\eqref{eq:local-uniform-boundary}.  If the boundary control value is a viscosity
solution of the same terminal-value problem, comparison identifies it with
$v_0$.  The proof is local near $r=0$ and is unchanged in the presence of a
second face $r=\ell(t,x)>0$.
\end{proof}

\subsection{Identification of the boundary value without dynamic programming}
\label{sec:perron}

The identification $v_0=\cV$ in
Theorem~\ref{thm:boundary-continuity} requires $\cV$ to belong to the
comparison class and to be a viscosity solution of
\eqref{eq:boundary-HJB} with terminal datum $\Phi(\cdot,0)$.
We provide sufficient conditions for this viscosity property by adapting
stochastic Perron's method \cite{BayraktarSirbu2013}, without assuming
a dynamic programming principle for the boundary problem.

Throughout this subsection, admissible boundary controls are bounded
predictable processes satisfying $a_s\in A_0(s,\bar X_s)$ for almost
every $s<T$, with a deterministic bound that may depend on the control.
Let $\cV_{\mathrm b}$ denote their value on a fixed probability space
and $\cV_{\mathrm w}$ the corresponding value over weak admissible
controls \cite[Definition 4.1]{BayraktarSirbu2013}. If the control
family in \eqref{eq:boundary-value} is larger, the results below
identify $\cV$ with these values provided that restricting to bounded
boundary controls does not change the value.

\begin{assumption}[Boundary feedback selections]\label{ass:selection}
The set $A\subset\RR^m$ is closed, and $A_0(t,x)$ is nonempty for $t<T$.
For every $(t_0,x_0)\in[0,T)\times\RR^d$
and $a_0\in A_0(t_0,x_0)$ there are a neighborhood $N$ of $(t_0,x_0)$ in
$[0,T)\times\RR^d$ and a measurable map $\bar a:N\to A$, continuous at $(t_0,x_0)$, with
$\bar a(t_0,x_0)=a_0$ and $\bar a(s,x)\in A_0(s,x)$ on $N$ such that, from
every stopping time $\tau$ and $\mathcal F_\tau$-measurable initial state in
$N$, the feedback equation \eqref{eq:X-boundary} with
$a^0_s=\bar a(s,\bar X_s)$ has a strong solution up to the exit time from
$N$.  Moreover, from every stopping time and
$\mathcal F_\tau$-measurable initial state there is an admissible boundary
control on $[\tau,T]$.
\end{assumption}

The selections imply lower semicontinuity of $A_0$ before $T$.  When $A_0$
has closed convex values, lower semicontinuity yields continuous selections
through every point by Michael's theorem \cite{Michael1956}.  A sufficient
condition for the local strong existence required here is local Lipschitz
continuity in $x$ of the composed coefficients
$b(s,x,0,\bar a(s,x))$ and $\sigma(s,x,0,\bar a(s,x))$, uniformly in $s$
on compact sets.  Assumption 2.10 of Bouchard, Gnecco Heredia, Moreau and
Pham \cite{BGHMP2026} supplies preterminal feedbacks in their model; their
Remark 2.11 gives additional sufficient conditions when the diffusion is
control-independent.  Finite pasting at stopping times preserves the
boundary constraint and boundedness, with the maximum of the finitely many
control bounds.

Fix the smooth growth function $\psi(x)=(1+|x|^2)^k$, $k\ge0$.  Weak
admissible controls are required to satisfy
$\EE\sup_s\psi(\bar X_s)<\infty$.  A function has growth $\psi$ if its
absolute value is bounded by a constant times $\psi$.  Write
$J^a_{\tau,\rho}:=\int_\tau^\rho f(s,\bar X_s,0,a_s)\,ds$.

\begin{definition}[Stochastic subsolutions and supersolutions]
\label{def:perron-classes}
A continuous function $v$ of growth $\psi$ with
$v(T,\cdot)\le\Phi(\cdot,0)$ is a stochastic subsolution if there is
$L(v)<\infty$ such that, for every stopping time $\tau$ and
$\mathcal F_\tau$-measurable $\xi$ with $\EE\psi(\xi)<\infty$, there is an
admissible boundary control on $[\tau,T]$, bounded by $L(v)$, for which
\[
 v(\tau,\xi)\le
 \EE\bigl[v(\rho,\bar X_\rho)+J^a_{\tau,\rho}\mid\mathcal F_\tau\bigr]
 \quad\hbox{for every stopping time }\rho\in[\tau,T].
\]
A continuous function $v$ of growth $\psi$ with
$v(T,\cdot)\ge\Phi(\cdot,0)$ is a stochastic supersolution if
\[
 v(\tau,\bar X_\tau)\ge
 \EE\bigl[v(\rho,\bar X_\rho)+J^a_{\tau,\rho}\mid\mathcal F_\tau\bigr]
 \quad\hbox{for all stopping times }\tau\le\rho
\]
along every weak admissible boundary control.
\end{definition}

These are the definitions of Bayraktar and S\^{\i}rbu
\cite[Sections 3 and 4]{BayraktarSirbu2013}, with running rewards included
as in their Remark 2.1.

\begin{theorem}[Stochastic Perron for the boundary problem]
\label{prop:perron}
Suppose Assumption \ref{ass:selection} holds, $H_0$ is finite and continuous
for $t<T$, and, for every compact $K\subset\RR^d$ and $Q<\infty$,
\[
 \sup_{\substack{0\le t<T,\ x\in K\\ |q|+|M|\le Q}}H_0(t,x,q,M)<\infty.
\]
Suppose
$b(\cdot,\cdot,0,\cdot)$, $\sigma(\cdot,\cdot,0,\cdot)$, $f(\cdot,\cdot,0,\cdot)$,
and $\Phi(\cdot,0)$ are continuous up to $T$, and $b,\sigma$ are Lipschitz
in $x$ with linear growth, uniformly for $t\in[0,T]$ and $a$ in bounded
sets.  Assume
\begin{equation}\label{eq:gauge}
 |\Phi(x,0)|\le C\psi(x),\qquad
 \sup_{|a|\le L}|f(t,x,0,a)|\le C_L\psi(x),
\end{equation}
and suppose both classes in Definition \ref{def:perron-classes} are
nonempty.  Suppose comparison holds
for \eqref{eq:boundary-HJB} with terminal datum $\Phi(\cdot,0)$ between upper
semicontinuous subsolutions and lower semicontinuous supersolutions of growth
$\psi$.  Then $\cV_{\mathrm b}=\cV_{\mathrm w}$ is the unique continuous
viscosity solution of growth $\psi$ of this terminal-value problem.  If
bounded boundary controls suffice in \eqref{eq:boundary-value}, this solution
is $\cV$ and, under the hypotheses of Theorem
\ref{thm:boundary-continuity}, $v_0=\cV$.
The same assertion holds on an open state domain $O$ when the admissible
processes remain in $O$, the stated growth and integrability conditions hold,
and comparison includes the relevant conditions on $\partial O$.
\end{theorem}

\begin{proof}
Let
\begin{align*}
 v^-&:=\sup\{v:\hbox{stochastic subsolution}\},\\
 v^+&:=\inf\{v:\hbox{stochastic supersolution}\}.
\end{align*}
Taking $\tau=t$, $\rho=T$
in the definitions gives $v^-\le\cV_{\mathrm b}\le\cV_{\mathrm w}\le v^+$.
Fix one member of each class.  They bound both envelopes from below and above,
so the envelopes are finite and have growth $\psi$.

The maximum of two stochastic subsolutions is a stochastic subsolution: the
pasted control of \cite[Proposition 3.1]{BayraktarSirbu2013} is admissible
with bound $\max\{L(v_1),L(v_2)\}$.  The minimum of two stochastic
supersolutions is one as well.  Hence $v^-$ and $v^+$ are monotone limits of
elements of the respective classes, and the Dini arguments of
\cite{BayraktarSirbu2013} apply.

\emph{$v^+$ is a viscosity subsolution.}  Suppose $v^+-\phi$ has a strict
local maximum equal to zero at $(t_0,x_0)$, $t_0<T$, with
$-\partial_t\phi-H_0(t_0,x_0,D\phi,D^2\phi)>0$.  By continuity of $H_0$
there is a ball $B$ around $(t_0,x_0)$ on which
$-\partial_t\phi-\mathcal B^a(s,x,0;D\phi,D^2\phi)>0$ for all
$a\in A_0(s,x)$.  Since every admissible boundary control satisfies
$a_s\in A_0(s,\bar X_s)$, It$\hat{o}$'s formula shows that
$\phi(s,\bar X_s)+J^a_{t,s}$ is a supermartingale up to the exit time from
$B$ for every weak admissible control.  This is the only property of the
control used in Step 2 of \cite[Theorem 4.1]{BayraktarSirbu2013}, which then
produces a stochastic supersolution strictly below $v^+$ at $(t_0,x_0)$, a
contradiction.  No selection is required in this step.

\emph{$v^-$ is a viscosity supersolution.}  Suppose $v^--\phi$ has a strict
local minimum equal to zero at $(t_0,x_0)$, $t_0<T$, with
$-\partial_t\phi-H_0(t_0,x_0,D\phi,D^2\phi)<0$.  Then some
$a_0\in A_0(t_0,x_0)$ satisfies
$-\partial_t\phi-\mathcal B^{a_0}(t_0,x_0,0;D\phi,D^2\phi)<0$.  Let
$\bar a$ be the selection of Assumption \ref{ass:selection} through $a_0$.
By continuity of the coefficients and of $\bar a$ at $(t_0,x_0)$ there is a
ball $B\subset N$ on which
$-\partial_t\phi-\mathcal B^{\bar a(s,x)}(s,x,0;D\phi,D^2\phi)<0$.  Given a
stochastic subsolution $v$ and an initial state in $B$, use the feedback
$\bar a(s,\bar X_s)$ until the exit time $\theta$ from $B$ and then the
control provided by $v$ from $(\theta,\bar X_\theta)$.  The result is an
admissible bounded boundary control, and It$\hat{o}$'s formula on $[\tau,\theta]$
gives the submartingale inequality for $\phi$.  This replaces the constant
control $\tilde u$ in Step 2 of \cite[Theorem 3.1]{BayraktarSirbu2013}; the
remaining construction of a stochastic subsolution strictly above $v^-$ at
$(t_0,x_0)$ is unchanged.

\emph{Terminal conditions.}  Suppose $v^-(T,x_0)<\Phi(x_0,0)$.  Choose
$c$ strictly between these values and $\delta,h>0$ such that
$\Phi(x,0)>c$ on $\overline B_\delta(x_0)$, and put
$Q=[T-h,T]\times\overline B_\delta(x_0)$.  Fix a stochastic subsolution
$v$ and set $\phi(t,x)=c-|x-x_0|^2/\eta-K(T-t)$.
First choose $\eta>0$ small enough that $\phi<v$ on the lateral boundary
of $Q$, and then choose $K$ large enough that the same inequality holds on
the bottom face $t=T-h$.  Increase $K$ further so that
\[
 \partial_t\phi+\mathcal B^a(t,x,0;D\phi,D^2\phi)\ge0
 \quad\hbox{on }Q\cap\{t<T\},\quad a\in A,\ |a|\le L(v).
\]
This is possible because the coefficients and reward are bounded on compact
state and control sets.  Define $\widetilde v=v\vee\phi$ on $Q$ and
$\widetilde v=v$ outside.  This function is continuous, has growth $\psi$,
and satisfies the terminal inequality.  On the event that $\phi$ is selected
at the starting point, use a control supplied by $v$ until exit from $Q$,
then concatenate a control supplied by $v$ at the exit point.  Both have
bound $L(v)$.  It$\hat{o}$'s formula for $\phi$ up to exit and the defining inequality
for $v$ afterward prove that $\widetilde v$ is a stochastic subsolution.
On the complementary event use the control of $v$ throughout.  This
contradicts $\widetilde v(T,x_0)\ge c>v^-(T,x_0)$.

If $v^+(T,x_0)>\Phi(x_0,0)$, choose $c$ strictly between these values and
a cylinder $Q$ on whose terminal face $\Phi<c$.  Fix a stochastic
supersolution $v$ and use
$\phi(t,x)=c+|x-x_0|^2/\eta+K(T-t)$.  Choose $\eta$ small and $K$ large
so that $\phi>v$ on the lateral and bottom faces.  The assumed upper bound
on $H_0$ near $T$ permits increasing $K$ so that
$-\partial_t\phi-H_0(t,x,D\phi,D^2\phi)\ge0$ on $Q\cap\{t<T\}$.
The function $v\wedge\phi$ on $Q$, extended by $v$ outside, is then a
stochastic supersolution by It$\hat{o}$'s formula and stopping at the exit time,
as in \cite[Theorem 4.1]{BayraktarSirbu2013}.  Its value at $(T,x_0)$ is
below $v^+(T,x_0)$, a contradiction.  Hence
$v^-(T,\cdot)\ge\Phi(\cdot,0)\ge v^+(T,\cdot)$.

\emph{Conclusion.}  Comparison gives $v^+\le v^-$, so
$v^-=\cV_{\mathrm b}=\cV_{\mathrm w}=v^+$ is upper and lower
semicontinuous, hence continuous, and has growth $\psi$.  It is the unique
viscosity solution in that growth class.  If bounded boundary controls
suffice, this solution is $\cV$, and the last clause of Theorem
\ref{thm:boundary-continuity} applies.  The arguments on an open state
domain are the same, using cylinders with spatial closure contained in $O$.
\end{proof}

\paragraph{Verifying nonemptiness}
A sufficient condition for a stochastic subsolution is the existence of
$L_0<\infty$ such that, from every stopping time and integrable initial
state as in Definition \ref{def:perron-classes}, an admissible boundary
control bounded by $L_0$ is available.  Linear growth of the coefficients
bounds the absolute value of the generator of $\psi$ by $C_{L_0}\psi$.
Together with \eqref{eq:gauge}, It$\hat{o}$'s formula then shows that
$-Ce^{K(T-t)}\psi(x)$ is a stochastic subsolution for fixed sufficiently
large $C,K$, independent of the starting time and state.
A stochastic supersolution is a constant plus a multiple of $T-t$ when
$\Phi(\cdot,0)$ and $\sup_{a\in A_0(t,x)}f(t,x,0,a)$ are bounded above.
More generally, $Ce^{K(T-t)}\psi$ works if
\[
 \sup_{a\in A_0(t,x)}
 \left\{b\cdot D\psi+\tfrac12\Tr(\sigma\sigma^\top D^2\psi)\right\}
 \le C_0\psi,\qquad
 \sup_{a\in A_0(t,x)}f\le C_0\psi.
\]
The moment condition and \eqref{eq:gauge} justify localization and passage
to the limit in these stochastic inequalities.

\paragraph{Dirichlet formulation}
Once $v_0=\cV$, Theorem \ref{thm:boundary-continuity} supplies the boundary
condition $u(t,x,0)=\cV(t,x)$ for the interior HJB equation stated in the
introduction.  If $V^*$ and $V_*$ satisfy the interior viscosity inequalities
and the terminal inequalities for $\Phi$, and comparison holds for that
Dirichlet problem in their growth class, then $V^*=V_*$ is its unique
continuous viscosity solution in that class.

\subsection{Bouchard, Elie and Imbert \texorpdfstring{\cite{BEI2010}}{}}\label{sec:BEI}

For a smooth target boundary, write
\[
 N^a=\sigma_Y^a-Dw\,\sigma_X^a,
 \qquad
 M^a=b_Y^a-L_X^a w.
\]
Here
\[
 L_X^a w=\partial_t w+b_X^a\cdot Dw
       +\frac12\Tr(\sigma_X^a(\sigma_X^a)^\top D^2w).
\]
Assumption 3.3 of Bouchard, Elie and Imbert \cite{BEI2010} requires the set
$\{a:N^a=0\}$ to be a singleton, represented by a locally Lipschitz control.
This singleton requirement belongs to their smooth Dirichlet subsection.  Their
earlier Theorem 3.1 and Corollary 3.1 allow a set of zero-volatility controls in
the constrained boundary operator, without identifying a lower-dimensional
Dirichlet value.
At points where their selected control $\widehat a$ lies in the interior of the
control set, the lower-bound part of their Assumption 3.4 is
\eqref{eq:bei-local-bound}, and their stochastic-target equation gives
$M^{\widehat a}=0$.  Thus
$A_0(t,x)=\{\widehat a(t,x)\}$ in the Dirichlet case considered here.

Theorem \ref{thm:boundary-continuity} uses the same local estimate and replaces
the singleton condition in that Dirichlet reduction by compactness and
continuity of the correspondence $A_0(t,x)$.  In the subsolution proof, the
penalization produces a sequence of controls.  Their Dirichlet proof uses
uniqueness to identify its limit; here compactness gives a limit in $A_0(t,x)$,
and the resulting boundary operator is the supremum over the full set.  Their bounded-control alternative
in Assumption 3.2 gives the compact case of Proposition
\ref{prop:compact-coercive}.  The next lemma records the calculation for their
unbounded-control alternative.

\begin{lemma}[Unbounded controls in Bouchard, Elie and Imbert \cite{BEI2010}]
\label{lem:BEI-coercivity}
Suppose $w\in C^{1,2}$ and the unbounded-control alternative in Assumption 3.2
of Bouchard, Elie and Imbert \cite{BEI2010} holds.  After the change of variables
$R=Y-w(t,X)$, the transformed coefficients satisfy Assumption
\ref{ass:coercive-controls}, with no running reward as in their model.
\end{lemma}

\begin{proof}
Fix a compact set of $(t,x)$ and let $0\le r\le r_0$.  The corresponding
resource levels $y=w(t,x)+r$ remain in a compact set.  Assumption 3.2 of
Bouchard, Elie and Imbert \cite{BEI2010} states, locally uniformly in $(x,y)$,
that
\begin{equation}\label{eq:BEI-unbounded-control}
 |\sigma_Y^a|\longrightarrow\infty,
 \qquad
 \frac{|b_Y^a|+|b_X^a|+|\sigma_X^a|^2}{|\sigma_Y^a|^2}
       \longrightarrow0
 \quad\hbox{as }|a|\to\infty.
\end{equation}
The derivatives of $w$ are bounded on the compact set.  Consequently,
\[
 \frac{|Dw\,\sigma_X^a|}{|\sigma_Y^a|}
 \le C\frac{|\sigma_X^a|}{|\sigma_Y^a|}\longrightarrow0.
\]
Since $\eta^a=\sigma_Y^a-Dw\,\sigma_X^a$, it follows that
$|\eta^a|/|\sigma_Y^a|\to1$, proving \eqref{eq:eta-coercive}.

The transformed surplus drift is
\[
 \mu^a=b_Y^a-\partial_t w-b_X^a\cdot Dw
       -\frac12\Tr(\sigma_X^a(\sigma_X^a)^\top D^2w).
\]
For bounded $(q,M)$, the part of the transformed generator that remains on the
boundary is
\[
 \mathcal B^a=b_X^a\cdot q
       +\frac12\Tr(\sigma_X^a(\sigma_X^a)^\top M).
\]
Equation \eqref{eq:BEI-unbounded-control} and
$|\eta^a|^2/|\sigma_Y^a|^2\to1$ therefore give, locally uniformly,
\[
 \frac{|\mu^a|+|\mathcal B^a|}{1+|\eta^a|^2}\longrightarrow0.
\]
This is \eqref{eq:coercive-ratio}.
\end{proof}

\begin{corollary}[Bouchard, Elie and Imbert \cite{BEI2010}]\label{cor:BEI}
In the smooth setting of Bouchard, Elie and Imbert \cite{BEI2010}, suppose their
Assumptions 3.2--3.4 hold and the selected control $\widehat a$ is interior at
every point.  More generally, suppose $M^{\widehat a}=0$ and the lower-bound
condition in their Assumption 3.4 holds at every point without the interior-control
qualification.  Suppose the terminal inequalities
in their Proposition 3.2 hold and the equation
\begin{equation}\label{eq:BEI-boundary-equation}
       -L_X^{\widehat a}v=0
\end{equation}
has comparison in the relevant growth class.  Then the upper and lower limits
of the constrained value on $y=w(t,x)$ coincide with the solution of
\eqref{eq:BEI-boundary-equation}.
\end{corollary}

\begin{proof}
The local Lipschitz property of $\widehat a$ verifies Assumption
\ref{ass:boundary-controls}.  Under either condition stated in the corollary,
the lower bound in their Assumption 3.4 gives \eqref{eq:local-lower-bound},
and their Assumption 3.2 gives compactness or, by Lemma
\ref{lem:BEI-coercivity}, Assumption \ref{ass:coercive-controls}, so
Proposition \ref{prop:compact-coercive} applies.  Their Theorem 6.1 gives
Assumption \ref{ass:basic}.  The subsolution clause of Assumption
\ref{ass:state-constraint-viscosity} is the interior inequality (3.9) of their
Theorem 3.1(2), which holds up to $y=w(t,x)$; under Assumption 3.2 the
strictly concave surplus term dominates for large controls, so the Hamiltonian
is continuous at the tests \eqref{eq:concave-surplus-tests} and a control can
be selected.  The boundary supersolution clause concerns the single control
$\widehat a$ and follows from their Theorem 6.1 with the Markov control
generated by the feedback $\widehat a$, for which $R=0$ is an equilibrium of
the surplus equation, as in the supersolution step of their Section 6.4; the
boundary operator of their Theorem 3.1(1) and Corollary 3.1 is not used.  Since
$A_0=\{\widehat a\}$, the Hamiltonian
\eqref{eq:boundary-Hamiltonian} is the operator in
\eqref{eq:BEI-boundary-equation}.  Theorem
\ref{thm:boundary-continuity} applies.
\end{proof}

\begin{remark}[Bouchard, Gnecco Heredia, Moreau and Pham \cite{BGHMP2026}]\label{rem:BGMP}
In \cite{BGHMP2026} the state $X$ has a compact control $u\in U$, the
constraint is $\EE[G(X_T)+\int_t^T g(X_s,u_s)\,ds]\le p$, and martingale
representation introduces a martingale $P$ with unbounded volatility
$\alpha\in\RR^d$.  Their boundary is $p=\varpi(t,x)+y$, with $y$ the
accumulated running cost and $\varpi$ the constraint value, a classical
solution of $\inf_{u\in U}\{L^u_X\varpi+g\}=0$ by their Proposition 2.8.  The surplus $R=P-\varpi(t,X)-Y$ then has
\[
 \mu^{u,\alpha}=-\bigl(L^u_X\varpi+g(x,u)\bigr)\le0,\qquad
 \eta^{u,\alpha}=\alpha-\sigma(x,u)^\top D\varpi(t,x),
\]
so Assumption \ref{ass:drift-cost} holds with $\mu^+=0$, and
$A_0(t,x)=\{(u,\sigma(x,u)^\top D\varpi(t,x)):u\in\widehat U(t,x)\}$, where
$\widehat U$ is the set of minimizers.  The coefficients of $X$ and the
rewards do not depend on $\alpha$ and $U$ is compact, so Assumption
\ref{ass:generator-growth} holds, and \eqref{eq:sequential-generator-bound}
follows by extracting a convergent subsequence of $u_n$ as in the proof of
Proposition \ref{prop:compact-coercive}.  Their Assumption 2.10 supplies
continuous boundary feedbacks through every control at times $t<T$, with
well-posed dynamics.  These give Assumption \ref{ass:boundary-controls},
continuity of $H_0$ before $T$, and the continuations for the boundary clause of Assumption
\ref{ass:state-constraint-viscosity}, whose subsolution clause follows from
their Theorem 2.5 since the supremum over $\alpha$ at the tests
\eqref{eq:concave-surplus-tests} is attained on a bounded set.  Under their
Assumption 2.10, their Theorem 2.13 is thus the instance of Theorem
\ref{thm:boundary-continuity} with nonpositive surplus drift; their terminal
estimates supply Assumption \ref{ass:terminal-limit}.  This application uses
the preterminal feedbacks and does not require selections at $T$.
\end{remark}

\subsection{Alvarez, Ekren and Huang \texorpdfstring{\cite{AEH2026}}{}}

The risk-neutral, $X$-independent reduction in Section 6 of
Alvarez, Ekren and Huang \cite{AEH2026} has state space
$\mathcal D=\{(t,y,p):\underline W(t)<y<\overline W(t),\ p\in(0,1)\}$.
Use the surplus $R_u=\overline W(t)-Y$ at the upper boundary and
$R_d=Y-\underline W(t)$ at the lower boundary.  The volatility of either surplus
is, up to sign, $z^0-z^1$ and therefore vanishes when $z^0=z^1$.  If
$\overline a$ and $\underline a$ denote the supremum and infimum of
$H^0(z)-H^1(z)$, the surplus drifts at $r=0$ are
\begin{align*}
 \mu_u(z^0,z^1)&=H^0(z^0)-H^1(z^1)-\overline a
   -\overline\lambda(p,z^0,z^1)(z^0-z^1),\\
 \mu_d(z^0,z^1)&=\underline a-H^0(z^0)+H^1(z^1)
   +\overline\lambda(p,z^0,z^1)(z^0-z^1).
\end{align*}
The equality $z^0=z^1$ leaves a continuum of controls.  The endpoint boundary
equation may therefore retain several controls, whereas Assumption 3.3 of Bouchard, Elie and
Imbert \cite{BEI2010} selects a singleton in their smooth Dirichlet setting.
Zero surplus drift and volatility require $H^0(z)-H^1(z)=\overline a$
or $H^0(z)-H^1(z)=\underline a$,
respectively.  Hence the two sets playing the role of $A_0$ are
\begin{align*}
 \mathcal A_u^0&=\{(z,z):H^0(z)-H^1(z)=\overline a\},\\
 \mathcal A_d^0&=\{(z,z):H^0(z)-H^1(z)=\underline a\}.
\end{align*}
The remaining state on either boundary is the belief $p$.  Optimizing its
generator and the reward over these sets gives exactly the operators
$\overline H$ and $\underline H$ in equation (54) of
Alvarez, Ekren and Huang \cite{AEH2026}, with their infimum convention corresponding
to the supremum convention in \eqref{eq:boundary-Hamiltonian}.

\begin{lemma}[Verification for Alvarez, Ekren and Huang \cite{AEH2026}]
\label{lem:AEH-verification}
Under Assumption 2 of Alvarez, Ekren and Huang \cite{AEH2026}, the coefficients
at each endpoint of the credible interval satisfy Assumptions
\ref{ass:drift-cost}, \ref{ass:generator-growth}, and
\ref{ass:generator-domination} after the change to $R_u$ or $R_d$.
\end{lemma}

\begin{proof}
Write $d=z^0-z^1$.  By their equation (45), the surplus volatility is
$\eta=\pm d$.  Their effort responses $A^0,A^1$ and
$\overline\lambda=pA^0(z^0)+(1-p)A^1(z^1)$ are bounded, while $H^0,H^1$ are
Lipschitz by their Lemma 3(i).  At the upper endpoint, the extremal inequality
$H^0(z)-H^1(z)\le\overline a$ gives the signed estimate
\[
 H^0(z^0)-H^1(z^1)-\overline a
 \le H^1(z^0)-H^1(z^1)\le C|d|.
\]
Consequently, including the $r$-dependent term in the surplus drift,
$\mu_u^+\le C(|d|+r)$.
The inequality $H^0(z)-H^1(z)\ge\underline a$, decomposed at $z^1$ in the
same way, gives $\mu_d\le C(|d|+r)$ and hence the same estimate for
$\mu_d^+$.  This proves \eqref{eq:drift-cost}, and Lemma
\ref{lem:drift-equivalence} shows that it is the local lower bound
\eqref{eq:local-lower-bound} at each endpoint.

The estimate is one-sided.  At the upper endpoint, the expression in equation
(72) of Alvarez, Ekren and Huang \cite{AEH2026} is
\[
 -nH^0(z^0)+nH^1(z^1)+n\overline a
       +n\overline\lambda d=-n\mu_u\big|_{r=0},
\]
and the preceding calculation bounds it from below by $-Cn|d|$.  Thus the
lower bound in their penalized inequality is precisely an upper bound on
$\mu_u$.  The lower endpoint has the corresponding inequality with the
opposite extremum.

The state that remains on either endpoint is $p$, whose volatility is
\[
 \gamma(p,z^0,z^1)=p(1-p)\bigl(A^0(z^0)-A^1(z^1)\bigr).
\]
It is bounded.  Their Lemma 3(ii), applied to the reward in their equation
(47), gives $|\ell|\le C(1+|d|)$.  Hence, for bounded boundary jets,
\[
 \left|\frac12\gamma^2 M+\ell\right|\le C(1+|d|),
\]
which implies \eqref{eq:generator-growth}.

Consider now a sequence for which $r_n\to0$, $\eta_n\to0$, and
$\mu_n\to0$.  Then $d_n\to0$, and the drift formulas imply
\[
 H^0(z_n^0)-H^1(z_n^0)\longrightarrow\overline a
 \quad\hbox{or}\quad
 H^0(z_n^0)-H^1(z_n^0)\longrightarrow\underline a,
\]
according to the endpoint.  Lemma 3(iii) of Alvarez, Ekren and Huang
\cite{AEH2026} implies that the distance from $z_n^0$ to the corresponding
endpoint level set tends to zero.  Indeed, otherwise a subsequence bounded away
from that set would have a further subsequence contradicting their lemma.
Choose $\widetilde z_n$ in the endpoint level set with
$|z_n^0-\widetilde z_n|\to0$.  Then
$(\widetilde z_n,\widetilde z_n)$ belongs to $\mathcal A_u^0$ or
$\mathcal A_d^0$.

The Lipschitz continuity of $A^0,A^1$ shows that the volatility of $p$ under
$(z_n^0,z_n^1)$ converges to that under
$(\widetilde z_n,\widetilde z_n)$.  The same is true of the reward.  To see the
latter without assuming the controls are bounded, first consider a subsequence
on which $z_n^0+z_n^1$ is bounded; ordinary continuity applies.  On a subsequence
where its absolute value tends to infinity, $d_n\to0$ places
$z_n^0,z_n^1$, and $\widetilde z_n$ on the same affine tail of $H^0,H^1$.
The effort responses then equal the same endpoint $\alpha$ of the compact effort
set, and equation (47) reduces to
\[
 \ell(t,z^0,z^1,p)
 =\alpha-\frac{e^{\kappa(T-t)}}2
       \bigl(c(0,\alpha)+c(1,\alpha)\bigr),
\]
independently of $z^0,z^1$, and $p$.  Thus
$\mathcal B^{(z_n^0,z_n^1)}-\mathcal B^{(\widetilde z_n,\widetilde z_n)}\to0$,
and since $(\widetilde z_n,\widetilde z_n)$ is a boundary control,
\eqref{eq:sequential-generator-bound} follows once $H_0$ is continuous.  On each affine tail outside a compact set of controls, the same
calculation reduces the endpoint operator to an expression independent
of the belief and the control level. Continuity on the remaining
compact set proves continuity of the two boundary Hamiltonians,
which are finite.
\end{proof}

\begin{corollary}[Alvarez, Ekren and Huang \cite{AEH2026}]\label{cor:AEH}
Under Assumption 2 of Alvarez, Ekren and Huang \cite{AEH2026}, the upper and
lower limits of their value at each endpoint of the credible interval satisfy
the subsolution and supersolution properties in their Proposition 4, including
the terminal value zero.  If their equations (56)--(57) satisfy comparison, the
value is continuous at both endpoints and has the boundary values in their
Theorem 4.
\end{corollary}

\begin{proof}
Apply the open-state-domain version of Theorem
\ref{thm:boundary-continuity} with $O=(0,1)$ for the belief
state, to $R_u$ and $R_d$ separately.  The other endpoint is a second face of the state space and
does not affect the local argument for $t<T$.  Lemma
\ref{lem:AEH-verification} verifies Assumptions \ref{ass:drift-cost},
\ref{ass:generator-growth}, and \ref{ass:generator-domination}.  Each boundary control set is constant,
nonempty, and contains finite controls; its belief dynamics have bounded
Lipschitz diffusion coefficients.  This verifies Assumption
\ref{ass:boundary-controls}.  Constant selections verify Assumption
\ref{ass:selection}.  A fixed finite endpoint control is admissible from
every initial belief in $(0,1)$ and supplies a uniform control bound.
The boundary reward is bounded, so the stochastic sub- and supersolutions
described after Theorem \ref{prop:perron} exist with $k=0$; the boundary
Hamiltonians are locally bounded for bounded jets up to $T$.  The affine-tail
calculation in Lemma \ref{lem:AEH-verification} also allows every boundary
control to be replaced by one in a fixed compact subset without changing
the belief dynamics or the reward.  Thus bounded controls suffice for the
endpoint values.  Under comparison, the open-state-domain version of Theorem
\ref{prop:perron} identifies each value with the solution of the corresponding
equation in (56)--(57).

Theorem 3 of Alvarez, Ekren and Huang \cite{AEH2026} gives the subsolution
inequality with the lower semicontinuous envelope of their Hamiltonian.
At a test \eqref{eq:concave-surplus-tests}, write $d=z^0-z^1$.
In their infimum convention, the expression being minimized is bounded below,
locally in the state and in the full test derivatives, by $c d^2-C(1+|d|)$
for some $c>0$.
Indeed, the $rr$ derivative contributes $\beta d^2$, the mixed term grows at
most linearly in $d$, and the drift and reward satisfy Lemma 3(ii) of
Alvarez, Ekren and Huang \cite{AEH2026}.  Near-minimizers therefore have
bounded $d$.  For bounded $d$, sufficiently large common values of $z^0,z^1$
place both controls on the same affine tail.  There the expression is
independent of their common level: $\gamma=0$ and the reward and surplus
drift do not depend on that level.  The infimum can consequently be taken over
a compact set, uniformly in a neighborhood of the test derivatives.
Continuity of the coefficients makes the Hamiltonian continuous there.
Their subsolution inequality thus gives the pointwise control selection in
Assumption \ref{ass:state-constraint-viscosity}.

The interior supersolution inequalities follow control by control from the
lower dynamic programming inequality in Proposition 3 of Alvarez, Ekren and
Huang \cite{AEH2026}.  We next verify the boundary clause.  Fix one endpoint and a boundary control
$(z,z)$ in the corresponding set $\mathcal A_u^0$ or $\mathcal A_d^0$.  If
$\psi$ touches the lower envelope from below at a boundary point, take interior
points converging to that contact and apply the lower dynamic programming
inequality in Proposition 3 of Alvarez, Ekren and Huang \cite{AEH2026} over a
time interval of length $h$, stopped on leaving the contact neighborhood, using
the control $(z,z)$.  For this diagonal control the surplus solves $dR_s=\kappa R_s\,ds$.
Choose $h<T-t_0$ and a neighborhood whose closure stays away from the other
endpoint.  From sufficiently small positive initial surplus, this control is
admissible until the local stopping time; there it is concatenated with a
feasible continuation.  The other endpoint causes no obstruction on this
stopped interval.  Let $\tau_h$ be $t_0+h$ stopped on leaving this
neighborhood.  Localization and passage to the boundary give
\[
 0\ge
 \EE\left[\int_{t_0}^{\tau_h}
 \left\{\partial_t\psi(s,p_s,0)
 +\mathcal B^{z,z}
   (s,p_s,0;D_p\psi,D^2_{pp}\psi)\right\}\,ds\right].
\]
Since $\EE[\tau_h-t_0]/h\to1$, division by $h$ and $h\downarrow0$ give
\[
 -\partial_t\psi(t_0,p_0,0)
 -\mathcal B^{z,z}(t_0,p_0,0;D_p\psi,D^2_{pp}\psi)\ge0.
\]
Since this holds for every boundary control, it is the boundary supersolution
inequality in Assumption \ref{ass:state-constraint-viscosity}.  The a priori value bounds in Proposition 2, equation (51), of
Alvarez, Ekren and Huang \cite{AEH2026}, together with the collapse of the
credible interval at $T$, give the terminal envelope condition.  The resulting
boundary Hamiltonians are their operators in equation (54).  Theorem
\ref{thm:boundary-continuity} therefore gives their Proposition 4, and
comparison gives the boundary assertion of their Theorem 4.
\end{proof}

\section{Verification from the Target Model}
\label{sec:original-variables}

This section gives a verification procedure from a stochastic target
model written in the variables $(X,Y)$. Suppose
\begin{align*}
 dX_s &= b_X(s,X_s,Y_s,a_s)\,ds+
        \sigma_X(s,X_s,Y_s,a_s)\,dW_s,\\
 dY_s &= b_Y(s,X_s,Y_s,a_s)\,ds+
        \sigma_Y(s,X_s,Y_s,a_s)\,dW_s,
\end{align*}
and let $w$ be the target value in \eqref{eq:w}. The first model-specific
step is to prove that the viable set is the closed epigraph
$y\ge w(t,x)$, including attainment of the target at $y=w(t,x)$.
One then verifies the controlled-process and viscosity hypotheses
as well as the coefficient conditions.

\paragraph{Smooth verification checklist}
For $w\in C^{1,2}$, the main theorem is verified as follows:
\begin{enumerate}[label=(\roman*)]
\item compute the surplus dynamics $R=Y-w(t,X)$, establish well-posedness and
dynamic programming in the chosen control formulation, and check local
boundedness of the value up to the boundary;
\item establish Assumption \ref{ass:state-constraint-viscosity}: justify selection
of a control at the concave tests \eqref{eq:concave-surplus-tests}, and verify
the lower dynamic programming inequality for admissible continuations of each
boundary control;
\item identify $A_0(t,x)$ by setting both the drift and volatility of $R$ equal
to zero at $R=0$;
\item verify the continuity of $A_0$ and Assumptions \ref{ass:drift-cost},
\ref{ass:generator-growth}, and \ref{ass:generator-domination}; for compact or
coercive controls the local lower bound \eqref{eq:bei-local-bound} suffices by
Proposition \ref{prop:compact-coercive}, and otherwise Lemma
\ref{prop:coefficient-set-test} gives a coefficient-set test;
\item form $H_0$ from the coefficients of $X$ and the reward on $A_0$, prove
comparison for the boundary equation \eqref{eq:boundary-HJB}, and identify
$\cV$ with its solution, for instance by Theorem \ref{prop:perron};
\item verify the terminal limit \eqref{eq:terminal-envelope-condition}, for
example through the state estimate \eqref{eq:terminal-state-estimate}.
\end{enumerate}

With $N^a$ and $M^a$ as in Section \ref{sec:BEI}, evaluated at
$(t,x,w(t,x),a)$, It$\hat{o}$'s formula shows that the volatility and drift of
$R=Y-w(t,X)$ on the boundary are $N^a$ and $M^a$.  Viability requires
$N^a=0$ and $M^a\ge0$, and the controls that keep the process on the boundary
form the set $A_0(t,x):=\{a:N^a(t,x)=M^a(t,x)=0\}$,
and the boundary coefficients are
$(b_X,\sigma_X,f)(t,x,w(t,x),a)$ for $a\in A_0(t,x)$.  If a control with
$N^a=0$ had $M^a>0$, the expression
$-nM^a+n^2|N^a|^2$ would tend to $-\infty$.  The local lower bound therefore
excludes such a strictly inward drift.  When
$w$ is a classical solution of the geometric dynamic programming equation of
Soner and Touzi \cite{SonerTouzi2002STP} with
$\sup\{M^a:N^a=0\}=0$, the drift is nonpositive at zero volatility.
The quantitative estimate needed for the boundary limit is
$(M^a)^+\le C(|N^a|+|N^a|^2)$, uniformly over controls on compact sets of
$(t,x)$, as in Lemma \ref{lem:drift-equivalence}.  It controls how positive
drift approaches zero as the volatility vanishes.  Section
\ref{sec:failure-lower-bound} shows that a smooth target value and the
pointwise condition at zero volatility can hold while this estimate fails.

Lower semicontinuity of $A_0$, or the selections of Assumption
\ref{ass:selection}, follow in concrete models from an explicit parametrization
of the solutions of $N=M=0$ or from an implicit function argument.  The
remaining coefficients give
\begin{equation}\label{eq:original-boundary-Hamiltonian}
 H_0(t,x,q,M)=\sup_{a\in A_0(t,x)}
 \left\{b_X\cdot q+\frac12\Tr(\sigma_X\sigma_X^\top M)+f\right\},
\end{equation}
with all coefficients evaluated at $(t,x,w(t,x),a)$.

\paragraph{Checking a proposed boundary value}
After computing $A_0$ and \eqref{eq:original-boundary-Hamiltonian}, a proposed
Dirichlet datum $g$ is checked as a viscosity solution of
$-\partial_t g-H_0(t,x,Dg,D^2g)=0$ with $g(T,\cdot)=\Phi(\cdot,0)$.  If $g$ lies
in the comparison class, Theorem \ref{thm:boundary-continuity} identifies it
with the limit of $V$; under the additional hypotheses of Theorem
\ref{prop:perron}, $g=\cV$.  For nonsmooth $w$, the change of variables based on It$\hat{o}$'s formula is
unavailable; the surplus dynamics and the state-constraint viscosity properties
must then be established directly before the theorem applies.

\section{Examples}\label{sec:examples}

The first three models verify Theorem \ref{thm:boundary-continuity} with a
continuum of zero-surplus-volatility controls generating distinct state drifts.
They therefore violate the singleton assumption in the smooth Dirichlet
reduction of Bouchard, Elie and Imbert \cite[Assumption 3.3(ii)]{BEI2010}.
Section \ref{sec:noncompact-controls} also violates their unbounded-control
condition \cite[(3.11)]{BEI2010}; a quadratic cost bounds the generators instead.
The degenerate state diffusions lie outside the uniform ellipticity assumption
of Bouchard, Gnecco Heredia, Moreau and Pham \cite[Theorem 2.13]{BGHMP2026}.
Section \ref{sec:noncompact-controls} adds an unbounded state control, and
Proposition \ref{prop:redundant-hedging} allows positive surplus drift.
These controlled drift and hedging problems extend the applications beyond
the principal-agent endpoint model of Alvarez, Ekren and Huang
\cite[Section 6]{AEH2026}.  Section \ref{sec:failure-lower-bound} shows that
boundary identification can fail if the local lower bound is omitted.

\subsection{A compact set-valued boundary-control example}
\label{sec:compact-controls}

Let $a\in[-1,1]$, $a^+=\max\{a,0\}$, and consider
\begin{gather*}
 dX_s=a_s\,ds,\qquad dR_s=a_s^+\,dW_s,\qquad R_s\ge0,
 \\
 V(t,x,r)=\sup_a\EE\left[\int_t^T a_s^+\,ds+G(X_T)\right],
\end{gather*}
where $G$ is bounded and uniformly continuous.  For the same dynamics with
$Y$ in place of $R$, the target $Y_T\ge0$ has value $w=0$: $a=0$ is feasible
from $y\ge0$, and the martingale property excludes $y<0$.

Here $\mu=0$ and $A_0=[-1,0]$, so the local lower bound and Proposition
\ref{prop:compact-coercive} apply.  Weak dynamic programming
\cite{BouchardNutz2012} and the continuous-Hamiltonian subsolution argument
of \cite[Theorem 3.1(2)]{BEI2010} give the upper viscosity inequality.
Each constant $a\in A_0$ keeps positive surplus constant, so lower dynamic
programming and It$\hat{o}$'s formula give the boundary supersolution inequality.
The bounds $|X_T-x|\le T-t$ and $\EE[R_T]=r$ verify the terminal limit.

The boundary equation
\[
 -\partial_t v-\sup_{a\in[-1,0]}aD_xv=0,
 \qquad v(T,x)=G(x),
\]
has bounded-solution comparison.  Constant selections, bounded coefficients,
and stochastic sub- and supersolutions $\mp\|G\|_\infty$ (using $a=0$ for
the subsolution) verify Theorem
\ref{prop:perron} with $k=0$.  Theorem \ref{thm:boundary-continuity} therefore
gives $V(t,x,r)\to\cV(t,x)$ locally uniformly as $r\downarrow0$,
with optimization over the whole interval $A_0$ retained in the boundary HJB.

\subsection{A noncompact-control example}
\label{sec:noncompact-controls}

Let $A=[-1,1]\times\RR$, write $a=(u,v)$, and consider
\[
 dX_s=(\tanh u_s+v_s)\,ds,\qquad dY_s=u_s\,dW_s,
 \qquad Y_T\ge0\quad\hbox{a.s.}
\]
Controls are progressively measurable and satisfy
$\EE\int_t^T|v_s|^2ds<\infty$.  The target value is $w=0$: a nonnegative
initial resource is feasible with $u=0$, while the martingale property of
$Y$ excludes a negative initial resource.  Set $R=Y$.  For a bounded uniformly
continuous function $G$, define
\begin{equation}\label{eq:noncompact-value}
 V(t,x,r)=\sup_{(u,v)}\EE\left[
 G(X_T)-\int_t^T|v_s|^2ds\right],\qquad R_s\ge0.
\end{equation}
The control $u=v=0$ is admissible, so $|V|\le\|G\|_\infty$.

Here
\begin{equation*}
 \mu^{u,v}=0,\qquad\eta^{u,v}=u,\qquad
 \mathcal B^{u,v}(q,M)=(\tanh u+v)q-v^2,\qquad
 A_0=\{0\}\times\RR.
\end{equation*}
For $|q|+|M|\le Q$,
\[
 (\mu^{u,v})^+=0,\qquad
 \mathcal B^{u,v}(q,M)\le |q|+\frac{q^2}{4}
                       \le Q+\frac{Q^2}{4}.
\]
Thus Assumptions \ref{ass:drift-cost} and \ref{ass:generator-growth} hold, and
\begin{equation*}
 H_0(q,M)=\sup_{v\in\RR}\{vq-v^2\}=\frac{q^2}{4}
\end{equation*}
is finite and continuous.  If $\eta^{u_n,v_n}\to0$ and $q_n\to q$, then
\[
 \mathcal B^{u_n,v_n}(q_n,M_n)
 =\tanh(u_n)q_n+\mathcal B^{0,v_n}(q_n,M_n)
 \le\tanh(u_n)q_n+H_0(q_n,M_n),
\]
and the right side tends to $H_0(q,M)$, which is
\eqref{eq:sequential-generator-bound}.
The correspondence $A_0$ is constant and its dynamics are well posed.
The full interior Hamiltonian is also finite and continuous:
\[
 H(q_x,q_r,N)=\frac{q_x^2}{4}
       +\sup_{u\in[-1,1]}\left\{\tanh(u)q_x+\frac12u^2N_{rr}\right\}.
\]
Weak dynamic programming \cite{BouchardNutz2012} and the constrained
subsolution argument of \cite[Theorem 3.1(2)]{BEI2010} apply; finiteness and
continuity of $H$ give the pointwise selection in Assumption
\ref{ass:state-constraint-viscosity}, and the boundary supersolution clause
follows as in Section \ref{sec:compact-controls}, since every $(0,v)$
keeps a positive surplus constant.

For the terminal limit, an $\epsilon$-optimal control in
\eqref{eq:noncompact-value} has $\EE\int_t^T|v_s|^2ds\le2\|G\|_\infty+\epsilon$,
hence $\EE|X_T-x|\le(T-t)+\sqrt{(T-t)(2\|G\|_\infty+\epsilon)}$ uniformly in
$r$; uniform continuity of $G$ gives the upper terminal inequality and
$u=v=0$ the lower one.

The boundary equation is
\begin{equation*}
 -\partial_t z-\frac14|D_xz|^2=0,\qquad z(T,x)=G(x),
\end{equation*}
for which the standard comparison principle applies to bounded semicontinuous
sub- and supersolutions.  Theorem \ref{prop:perron} applies with $k=0$:
$A$ is closed and $A_0$ is constant, so Assumption \ref{ass:selection} holds
with constant selections and the dynamics $d\bar X_s=v\,ds$; the reward
satisfies $|f(0,v)|=v^2\le L^2$ for $|v|\le L$, although it is unbounded on
$A_0$; and $G$ and $\sup_{A_0}f=0$ are bounded above, so a constant is a
stochastic supersolution.  The constant $-\|G\|_\infty$ is a stochastic
subsolution using $(u,v)=(0,0)$, and $H_0$ is bounded on bounded jets up to
$T$.  Thus the bounded-control boundary value is the unique bounded
viscosity solution.  It agrees with the value over all square-integrable
velocity controls: pathwise Cauchy--Schwarz bounds the displacement cost
below by $|X_T-x|^2/(T-t)$, and every fixed terminal displacement is attained
by a constant velocity.  Consequently,
\begin{equation}\label{eq:noncompact-boundary-value}
 \cV(t,x)=\sup_{y\in\RR}
       \left\{G(y)-\frac{|y-x|^2}{T-t}\right\},\qquad t<T.
\end{equation}
Equation \eqref{eq:noncompact-boundary-value} identifies $\cV$ with that
viscosity solution.  Theorem \ref{thm:boundary-continuity} gives locally
uniform convergence of $V(t,x,r)$ to $\cV(t,x)$ as $r\downarrow0$.

Assumption \ref{ass:coercive-controls} fails along $(u_n,v_n)=(0,n)$, where
$|a_n|\to\infty$ but $\eta^{a_n}=0$; the quadratic cost supplies
\eqref{eq:generator-growth} instead.

\subsection{Redundant hedging instruments}

A nonconstant target boundary also arises from redundant hedging instruments.
Consider two trading channels for the same discounted risk exposure, whose
price carries an excess return.  Their aggregate position determines wealth,
while their allocation changes a second state that records cumulative
imbalance between the channels.  The excess return makes positive surplus
drift available at a volatility cost, so this is the example in which the
drift condition is active.

\begin{proposition}[Redundant hedging instruments with controlled allocation]
\label{prop:redundant-hedging}
Let $P_s^{t,p}=p+\vartheta(s-t)+W_s-W_t$ for a constant excess return
$\vartheta\in\RR$, and let $\Psi\in C_b^3(\RR)$ satisfy $\Psi'\ge0$ and
$\Psi'\not\equiv0$.  Set
\begin{equation}\label{eq:redundant-target-value}
       w(t,p)=\EE\bigl[\Psi(p+W_T-W_t)\bigr],
\end{equation}
which does not depend on $\vartheta$, and choose $M>\|\Psi'\|_\infty$.  A control
$\pi=(\pi^1,\pi^2)$ takes values in $A=[-M,M]^2$ and generates discounted
wealth and an allocation state according to
\begin{equation*}
       dY_s=(\pi_s^1+\pi_s^2)\,dP_s,
       \qquad
       dQ_s=(\pi_s^1-\pi_s^2)\,ds.
\end{equation*}
Then $w(t,p)$ is the minimal initial wealth for the target
$Y_T\ge\Psi(P_T)$, independently of the initial allocation state $Q_t=q$.

Fix $c_1,c_2>0$, a bounded uniformly continuous function $k$, and bounded
uniformly continuous functions $G,\zeta$, with $\zeta(0)=0$.  For
$R_s=Y_s-w(s,P_s)$ and $R_t=r$, consider
\begin{equation*}
 V(t,p,q,r)=\sup_{\pi}\EE\left[
   -\int_t^T\bigl(c_1|\pi_s^1|^2+c_2|\pi_s^2|^2+k(Q_s)\bigr)\,ds
   +G(P_T,Q_T)+\zeta(R_T)\right],
\end{equation*}
where the supremum is over progressively measurable controls in $A$ satisfying
$R_s\ge0$ for $s\in[t,T]$.  The controls that keep the process on $R=0$ form
the line segment
\begin{equation}\label{eq:redundant-A0}
 A_0(t,p,q)=
 \left\{\pi\in[-M,M]^2:
          \pi^1+\pi^2=D_pw(t,p)\right\}.
\end{equation}
The corresponding boundary control problem is
\begin{equation}\label{eq:redundant-boundary-value}
 \cV(t,p,q)=
 \sup_{\substack{\pi_s\in A_0(s,P_s,Q_s)}}
 \EE\left[
   -\int_t^T\bigl(c_1|\pi_s^1|^2+c_2|\pi_s^2|^2+k(Q_s)\bigr)\,ds
   +G(P_T,Q_T)\right],
\end{equation}
with boundary dynamics
\[
       dP_s=\vartheta\,ds+dW_s,\qquad dQ_s=(\pi_s^1-\pi_s^2)\,ds.
\]
Then, locally uniformly in $(t,p,q)$,
\begin{equation}\label{eq:redundant-boundary-limit}
       \lim_{r\downarrow0}V(t,p,q,r)=\cV(t,p,q).
\end{equation}
The target boundary $w$ is nonconstant, and controls in
$A_0(t,p,q)$ generate different drifts of $Q$.  For $\vartheta\ne0$ the
surplus has drift $\vartheta\eta$, so departures from the replicating hedge
can earn positive surplus drift at the price of surplus volatility, and the drift condition
\eqref{eq:drift-cost} is nontrivial.  For example, with
$\Psi(p)=(1+\tanh p)/2$ and $G(p,q)=\tanh q$, both the target boundary and the
endogenous Dirichlet datum are nonconstant.
\end{proposition}

\begin{proof}
The function $w$ belongs to $C_b^{1,2}$ and solves
\begin{equation}\label{eq:redundant-heat}
       \partial_tw+\frac12D^2_{pp}w=0,
       \qquad w(T,\cdot)=\Psi.
\end{equation}
Moreover,
\[
       0\le D_pw(t,p)
       =\EE\bigl[\Psi'(p+W_T-W_t)\bigr]
       \le\|\Psi'\|_\infty<M.
\]
The control $\pi^1_s=D_pw(s,P_s)$, $\pi^2_s=0$ is admissible and keeps
$Y_s-w(s,P_s)$ constant, so it attains the target from every $y\ge w(t,p)$;
conversely, under the equivalent measure $\mathbb Q$ for which $P$ is a
Brownian motion, $Y$ is a bounded-volatility martingale, so $Y_T\ge\Psi(P_T)$
forces $y=\EE^{\mathbb Q}[Y_T]\ge\EE^{\mathbb Q}[\Psi(P_T)]=w(t,p)$.  Thus \eqref{eq:redundant-target-value} is the target
value and the viable set is $\{y\ge w(t,p)\}$.

For an arbitrary control, \eqref{eq:redundant-heat} and It$\hat{o}$'s formula give
\begin{equation*}
       dR_s=
       \bigl(\pi_s^1+\pi_s^2-D_pw(s,P_s)\bigr)\,(\vartheta\,ds+dW_s).
\end{equation*}
For the state $X=(P,Q)$ the surplus coefficients are therefore
$\eta=\pi^1+\pi^2-D_pw(t,p)$ and $\mu=\vartheta\eta$, which proves
\eqref{eq:redundant-A0}.  Since $|D_pw|<M$, this set is a
nondegenerate segment: with $z=D_pw(t,p)$ and $v=\pi^1-\pi^2$, which
ranges over $[-2M+|z|,2M-|z|]$, one has $\pi^1=(z+v)/2$ and
$\pi^2=(z-v)/2$, so the slices vary continuously in the Hausdorff metric and
Assumption
\ref{ass:boundary-controls} holds.  Assumption \ref{ass:drift-cost} holds with $C=|\vartheta|$, since
$\mu^+\le|\vartheta||\eta|$; by Lemma \ref{lem:drift-equivalence} this is the
local lower bound \eqref{eq:local-lower-bound}.  For sufficiently large
$\lambda$, the infimum of $-\lambda\mu+\lambda^2|\eta|^2$ equals
$-\vartheta^2/4$, attained at the feasible value $\eta=\vartheta/(2\lambda)$.

The coefficients and rewards are bounded and continuous and the control set
is compact.  The admissible families are stable under conditioning and
concatenation, and weak dynamic programming applies as in Bouchard and Nutz
\cite{BouchardNutz2012}.  The constrained subsolution argument of Bouchard,
Elie and Imbert \cite[Theorem 3.1(2)]{BEI2010} gives the upper inequality
in Assumption \ref{ass:state-constraint-viscosity}; the compact-control
Hamiltonian is continuous.  For the boundary supersolution assertion, fix a boundary control at
$(t_0,p_0,q_0)$ with parameter $\theta=(\pi^1-\pi^2)/(2M-|D_pw(t_0,p_0)|)$ and
use the feedback $v_s=(2M-|D_pw(s,P_s)|)\theta$,
$\pi^{1,2}_s=(D_pw(s,P_s)\pm v_s)/2$, which keeps $R$ constant from every
positive surplus and has well-posed dynamics; lower dynamic programming and
It$\hat{o}$'s formula on a shrinking stopped interval give the inequality for that
control.  Under every admissible control $R_T=r+\int_t^T\eta_s(\vartheta\,ds+dW_s)$
with $|\eta|\le3M$, and the drifts of $P,Q$ are bounded, so
$\EE[|P_T-p|+|Q_T-q|+R_T]\le C\sqrt{T-t}+C(T-t)+r$ and Assumption
\ref{ass:terminal-limit} holds.

The boundary value \eqref{eq:redundant-boundary-value} is the bounded
viscosity solution of
\begin{align}
 &-\partial_t v-\vartheta D_pv-\frac12D^2_{pp}v \nonumber\\
 &\qquad-\sup_{\pi\in A_0(t,p,q)}
 \bigl\{(\pi^1-\pi^2)D_qv
       -c_1|\pi^1|^2-c_2|\pi^2|^2-k(q)\bigr\}=0,
 \label{eq:redundant-boundary-HJB}\\
 &v(T,p,q)=G(p,q). \nonumber
\end{align}
To verify comparison, put $z=D_pw(t,p)$ and parameterize the controls by
a fixed parameter $\theta\in[-1,1]$:
\[
 v=(2M-|z|)\theta,\qquad
 \pi^1=(z+v)/2,\qquad \pi^2=(z-v)/2.
\]
Since $D_pw$ is bounded and Lipschitz in $p$, the drift of $Q$ is uniformly
Lipschitz in $(p,q)$, the drift and diffusion of $P$ are constant, and the running reward
is bounded and uniformly continuous.  The fixed compact control representation
therefore gives the usual bounded-solution comparison principle for
\eqref{eq:redundant-boundary-HJB}, and the same parametrization supplies the
selections of Assumption \ref{ass:selection}, so Theorem
\ref{prop:perron} applies with $k=0$: all controls are bounded, any fixed
$\theta$ gives an admissible boundary feedback from every initial state,
and bounded rewards give stochastic sub- and supersolutions of the form
$\pm C(1+T-t)$.  The boundary Hamiltonian is bounded on bounded jets up to
$T$.  Hence $\cV$ is the viscosity solution.  Proposition
\ref{prop:compact-coercive} and Theorem \ref{thm:boundary-continuity} now give
\eqref{eq:redundant-boundary-limit}.

The assumptions on $\Psi$ make $w$ nonconstant.  If
$\Psi(p)=(1+\tanh p)/2$ and $G(p,q)=\tanh q$, then the terminal value of $\cV$
is nonconstant.  Continuity at $T$ makes $\cV(t,\cdot,\cdot)$ nonconstant for
$t$ sufficiently close to $T$.
\end{proof}

\subsection{A smooth target without the local lower bound}
\label{sec:failure-lower-bound}

Let there be no $x$ variable and consider the target problem
\[
 dY_s=a_s\,ds+a_s^2\,dW_s,\qquad a_s\in[0,1],
 \qquad Y_T\ge0\quad\hbox{a.s.}
\]
Its target value is $w(t)=0$.  The control $a=0$ makes every $y\ge0$
feasible.  Conversely, for $\lambda>0$ let $\mathbb Q^\lambda$ be the
equivalent measure under which
$W_s-W_t+\lambda(s-t)$ is Brownian on $[t,T]$.  Boundedness of $a$ gives
\[
 0\le\EE^{\mathbb Q^\lambda}[Y_T]
 =y+\EE^{\mathbb Q^\lambda}\int_t^T(a_s-\lambda a_s^2)\,ds
 \le y+\frac{T-t}{4\lambda}
\]
for every feasible control.  Letting $\lambda\to\infty$ proves $y\ge0$.
Thus the viable epigraph is $Y\ge0$, and $R=Y$ is its surplus coordinate.
The coefficients satisfy $\mu(a)=a$ and $\eta(a)=a^2$, so $A_0=\{0\}$ and
$\mu=0$ whenever $\eta=0$.  Nevertheless, for $\lambda\ge1/4$,
\[
 \inf_{a\in[0,1]}\{-\lambda a+\lambda^2 a^4\}
 =-\frac{3}{4^{4/3}}\lambda^{2/3}\longrightarrow-\infty.
\]
The local lower bound therefore fails despite smoothness of $w$.

For the continuation problem, impose $R_s\ge0$, zero running reward, and
$\Phi(r)=r/(1+r)$.  The boundary-preserving value is $\cV(t)=0$.  To obtain
an admissible control with a positive payoff even from $R_t=0$, let
\[
 dZ_s=(\tfrac12-\tfrac18 Z_s)\,ds+\tfrac12 Z_s\,dW_s,
 \qquad Z_t=0,
 \qquad \tau=\inf\{s\ge t:Z_s=1\}\wedge T.
\]
The explicit solution is
$Z_s=\tfrac12 E_s\int_t^s E_u^{-1}\,du$, where
$E_s=\exp((W_s-W_t)/2-(s-t)/4)$, so $Z_s>0$ for $s>t$.
Set $a_s=Z_s\mathbf1_{\{s\le\tau\}}$.  This control takes values in
$[0,1]$, and It$\hat{o}$'s formula gives, for every initial surplus $r\ge0$,
\[
 R_s=r+Z_{s\wedge\tau}^2\ge0,
 \qquad dR_s=a_s\,ds+a_s^2\,dW_s.
\]
In particular, the target boundary is not absorbing.  Since $Z_\tau>0$
almost surely for $t<T$, bounded convergence yields
\[
 \liminf_{r\downarrow0}V(t,r)
 \ge\EE\left[\frac{Z_\tau^2}{1+Z_\tau^2}\right]>0=\cV(t).
\]
When the local lower bound fails, the interior value can therefore exceed the
value obtained by controls that keep the process on the boundary, even when
the target value is smooth.

The lower bound is sufficient, not necessary: with zero terminal payoff both
values vanish although it still fails.

\section{Conclusion}\label{sec:conclusion}

The target-generated change of variables converts a viable epigraph into a
half-space state constraint and turns the controls that preserve its boundary
into a boundary control problem.  Theorem
\ref{thm:boundary-continuity} proves that the value converges to the solution of
the resulting boundary HJB under one set of coefficient conditions.  At zero
surplus the drift condition is equivalent to the local lower bound of
Bouchard, Elie and Imbert \cite{BEI2010}; the signed estimates of Alvarez,
Ekren and Huang \cite{AEH2026} imply it near each endpoint.  Proposition
\ref{prop:compact-coercive} verifies the coefficient conditions for compact
or coercive controls under the local lower bound and continuity of $A_0$.
Direct generator estimates also cover the noncompact endpoint problems of
\cite{AEH2026} and the expectation-constraint model of \cite{BGHMP2026}.

\appendix
\section{Terminal limit}\label{app:terminal-limit}

\begin{lemma}[Terminal limit from a state estimate]\label{lem:terminal-state-limit}
Suppose the admissible control families are nonempty, $f$ and $\Phi$ are
bounded, and $\Phi$ is uniformly continuous.  Then
\eqref{eq:terminal-state-estimate} implies
\eqref{eq:terminal-envelope-condition}.
\end{lemma}

\begin{proof}
Fix $x$.  Let $m_\delta(x)\to0$ be the supremum in
\eqref{eq:terminal-state-estimate} and $\omega_\Phi$ a modulus of continuity
of $\Phi$.  For a starting point $(t,y,r)$ in that supremum and an admissible
control $a$, write $D^a=|X_T^{t,y,r,a}-x|+R_T^{t,y,r,a}$ and let
$J^a(t,y,r)$ be its expected payoff.  Since $R_T\ge0$, splitting the terminal
payoff according to $D^a\le\varepsilon$ and using Markov's inequality gives
\begin{align*}
 |J^a(t,y,r)-\Phi(x,0)|
 &\le\delta\|f\|_\infty+\omega_\Phi(\varepsilon)
       +2\|\Phi\|_\infty\PP(D^a>\varepsilon)\\
 &\le\delta\|f\|_\infty+\omega_\Phi(\varepsilon)
       +\frac{2\|\Phi\|_\infty}{\varepsilon}m_\delta(x).
\end{align*}
Nonemptiness of the control family gives the same bound for
$|V(t,y,r)-\Phi(x,0)|$.  Letting $\delta\downarrow0$ and then
$\varepsilon\downarrow0$ proves
$\overline V(T,x)=\underline V(T,x)=\Phi(x,0)$.
\end{proof}

\bibliographystyle{siamplain}
\bibliography{v5}

\end{document}